\documentclass[11pt,draft]{amsart}

\usepackage[english]{babel}
\usepackage{amsmath}
\usepackage{amsfonts}
\usepackage{amssymb}
\usepackage{amsthm}
\usepackage{epsfig}
\usepackage{graphicx,colortbl}
\usepackage{url,hyperref}

\newtheorem{theorem}{Theorem}[section]
\newtheorem{lemma}{Lemma}[section]
\newtheorem{corollary}{Corollary}[section]
\newtheorem{proposition}{Proposition}[section]
\newtheorem{definition}{Definition}[section]

\newtheorem{remark}{Remark}[section]
\newtheorem{example}{Example}[section]

\newcounter{theor}

\newtheorem{theo}[theor]{Theorem}

\DeclareMathOperator{\inter}{int}

\def\spann{\mathop\mathrm{span}\nolimits}
\def\conv{\mathop\mathrm{conv}\nolimits}
\def\bd{\mathop\mathrm{bd}\nolimits}

\def\s{\mathbb{S}}
\def\K{\mathcal{K}}
\def\w{\mathcal{W}}
\def\L{\mathcal{L}}

\def\R{\mathbb{R}}

\def\Z{\mathbb{Z}}
\def\V{\mathrm{V}}
\def\G{\mathrm{G}}
\def\I{\mathrm{I}}
\def\SL{\mathrm{SL}}
\def\vol{\mathrm{vol}}

\def\co{\mathfrak{C}}

\def\b{\mathrm{b}}
\def\W{\mathrm{W}}
\def\Wk#1#2{\mathrm{W}^{^{(#1)}}_{#2}\!}
\def\wk#1#2{\mathcal{W}^{^{(#1)}}_{#2}\!}
\def\m{\mathfrak{m}}
\def\cir{\mathrm{R}}
\def\e{\mathrm{e}}

\def\u{\pi(\cdot)^2}
\def\esc#1{\left\langle #1\right\rangle}

\def\f#1#2#3{f_{_{\!#1,#2}}^{#3}}
\newcommand{\dlat}{\mathrm{d} }

\numberwithin{equation}{section}

\begin{document}

\title{Further inequalities for the (generalized) Wills functional}

\author[D. Alonso-Guti\'errez]{David Alonso-Guti\'errez}
\address{Departamento de Matem\'aticas, Universidad de Zaragoza, C/ Pedro Cerbuna 12,
50009-Zaragoza, Spain} \email{alonsod@unizar.es}

\author[M. A. Hern\'andez Cifre]{Mar\'\i a A. Hern\'andez Cifre}
\author[J. Yepes Nicol\'as]{Jes\'us Yepes Nicol\'as}
\address{Departamento de Matem\'aticas, Universidad de Murcia, Campus de
Espinar\-do, 30100-Murcia, Spain}
\email{mhcifre@um.es}\email{jesus.yepes@um.es}

\thanks{The work is partially supported by MICINN/FEDER
projects MTM2016-77710-P and PGC2018-097046-B-I00, by DGA E26\_17R
and by ``Programa de Ayudas a Grupos de Excelencia de la Regi\'on
de Murcia'', Fundaci\'on S\'eneca, 19901/GERM/15}



\begin{abstract}
The Wills functional $\w(K)$ of a convex body $K$, defined as the sum of
its intrinsic volumes $\V_i(K)$, turns out to have many interesting
applications and properties. In this paper we make profit of the fact that
it can be represented as the integral of a log-concave function, which,
furthermore, is the Asplund product of other two log-concave functions,
and obtain new properties of the Wills functional (indeed, we will work in
a more general setting). Among others, we get upper bounds for $\w(K)$ in
terms of the volume of $K$, as well as Brunn-Minkowski and Rogers-Shephard
type inequalities for this functional. We also show that the cube of
edge-length 2 maximizes $\w(K)$ among all $0$-symmetric convex bodies in
John position, and we reprove the well-known McMullen inequality
$\w(K)\leq e^{\V_1(K)}$ using a different approach.
\end{abstract}

\maketitle

\section{Introduction, notation and main results}

Let $\K^n$ be the set of all convex bodies, i.e., non-empty compact convex
sets, in the $n$-dimensional Euclidean space $\R^n$. Let
$\langle\cdot,\cdot\rangle$ and $|\cdot|$ be the standard inner product
and the Euclidean norm in $\R^n$. For $K\in\K^n$ containing the origin in
its interior, $|x|_K=\min\{\lambda\geq 0:x\in\lambda K\}$ will denote the
well-known Minkowski functional of $K$, which, in the case when $K$ is
$0$-symmetric (i.e., $K=-K$), defines a norm whose unit ball is $K$.

We represent by $B^n_2$ the $n$-dimensional Euclidean (closed) unit ball
and by $\s^{n-1}$ its boundary. The volume of a measurable set
$M\subset\R^n$, i.e., its $n$-dimensional Lebesgue measure, is denoted by
$\vol(M)$ (or $\vol_n(M)$ if the distinction of the dimension is useful).
In particular, we write $\kappa_n=\vol(B^n_2)$, which takes the value
\begin{equation*}\label{e:kappa_n}
\kappa_n=\frac{\pi^{n/2}}{\Gamma\left(\frac{n}{2}+1\right)},
\end{equation*}
where $\Gamma(s)=\int_0^{\infty}t^{s-1}e^{-t}\,\dlat t$, for $s>0$,
represents the Gamma function.

The Grassmannian of $k$-dimensio\-nal linear subspaces of $\R^n$ is
denoted by $\G(n,k)$, and for $H\in\G(n,k)$, the orthogonal projection of
$M$ onto $H$ is denoted by $P_HM$, whereas the orthogonal complement of
$H$ is represented by $H^{\bot}$. Moreover, for $v\in\s^{n-1}$ and
$t\in\R$, we write $H_{v,t}=\bigl\{x\in\R^n:\langle x,v\rangle=t\bigr\}$.
With $\inter M$, $\bd M$ and $\conv M$ we denote the interior, boundary
and convex hull of $M$, respectively. Finally, as usual in the literature,
$\SL(n)$ stands for the subgroup of volume-preserving and
orientation-preserving linear transformations.

For convex bodies $K,E\in\K^n$ and a non-negative real number
$\lambda$, the well-known \emph{Steiner formula} states that the
volume of the Minkowski sum $K+\lambda E$ can be expressed as a
polynomial of degree (at most) $n$ in the parameter $\lambda$,
\begin{equation}\label{e:Steinerform}
\vol(K+\lambda E)=\sum_{i=0}^{n}\binom{n}{i}\W_i(K;E)\lambda^i;
\end{equation}
here, $\W_i(K;E)$ are the \emph{relative quermassintegrals} of $K$ with
respect to $E$, and they are a special case of the more general defined
mixed volumes (see e.g. \cite[Section~5.1]{Sch}). In particular,
$\W_0(K;E)=\vol(K)$ and $\W_n(K;E)=\vol(E)$. As usual in the literature,
we shorten $\W_i(K;B^n_2)=\W_i(K)$. In this case we observe that if
$K\in\K^n$ has dimension $\dim K=k$, then we can obtain the $i$-th
quermassintegral $\W_i(K)$ in $\R^n$, but also its $i$-th quermassintegral
computed in the $k$-dimensional affine subspace where $K$ is contained
(identified with $\R^k$), which we denote by $\Wk{k}{i}(K)$,
$i=0,\dots,k$. These two numbers do not coincide; indeed we have (see e.g.
\cite[Property~3.1]{SY93})
\[
\Wk{k}{i}(K)=\frac{\binom{n}{n-k+i}}{\binom{k}{i}}\frac{\kappa_i}{\kappa_{n-k+i}}\W_{n-k+i}(K),\quad
i=0,\dots,k,
\]
with $\W_i(K)=0$ for $i=0,\dots,n-k-1$. In order to avoid the
issue that quermassintegrals depend on the space where the convex
body is embedded, McMullen \cite{Mc75} defined the {\it intrinsic
volumes} of a convex body $K\in\K^n$ as
\[
\V_i(K)=\frac{\binom{n}{i}}{\kappa_{n-i}}\W_{n-i}(K),\quad
i=0,\dots,n.
\]
The particular case $i=1$ defines another well-known functional associated
to a convex body: the {\it mean width} $\b(K)$ of $K$. More precisely,
\[
\V_1(K)=\dfrac{n}{\kappa_{n-1}}\W_{n-1}(K)=\dfrac{n\kappa_n}{2\kappa_{n-1}}\b(K),
\]
and so it can be expressed in terms of the support function $h_K$ of $K$
as
\begin{equation}\label{e:mean_width}
\V_1(K)=\dfrac{1}{\kappa_{n-1}}\int_{\s^{n-1}}h_K(u)\,\dlat u
\end{equation}
(\cite[page~50 and (5.57)]{Sch}); here $\dlat u$ stands for the Lebesgue
measure on $\s^{n-1}$. We recall that the {\it support function} of $K$,
$h_K(u)=\max\bigl\{\langle z,u\rangle:z\in K\bigr\}$ for $u\in\s^{n-1}$,
is a convex function that uniquely determines the convex body (see e.g.
\cite[Section~1.7]{Sch}).

In 1973 (see \cite{W73}) Wills introduced and studied the functional
\begin{equation}\label{e:FWills}
\w(K)=\sum_{i=0}^n\binom{n}{i}\frac{\W_i(K)}{\kappa_i}=\sum_{i=0}^n\V_i(K)
\end{equation}
because of \emph{its possible relation} with the so-called
\emph{lattice-point enumerator} $G(K)=\#(K\cap\Z^n)$, and
conjectured that $\w(K)$ bounded by above $G(K)$. Although
Hadwiger \cite{Ha79} showed that Wills' conjecture was wrong (see
also \cite{BH93}), the Wills functional turned out to have several
interesting applications, e.g., in Discrete Geometry, where there
exist nice relations of this functional with the so-called
\emph{successive minima} of a convex body \cite{W90}, or in
deriving exponential moment inequalities for Gaussian random
processes \cite{Vi96}; see also \cite{Vi99,Vi01,Vi19}. It has also
been studied in other probabilistic context, \cite{Vi08}, and the
behavior of its roots has been analyzed when it is seen as a
formal polynomial in a complex variable, \cite{HCY13,HCY15,W75}.

The Wills functional exhibits many nice and engaging properties. For
instance, observe that $\w(K)$ depends only on the convex body, but not on
the dimension of the embedding space. In the next theorem we enumerate
some properties of $\w(K)$.
\begin{theo}\label{t:properties_Wills}
Let $H\in\G(n,k)$ and let $K,E\in\K^n$.
\begin{enumerate}
\item Hadwiger \cite[(2.3)]{Ha75}: if $K\subset H$ and $E\subset H^{\bot}$,
then
\[
\w(K)\w(E)=\w(K+E).
\]
%\item Hadwiger \cite[(2.7)]{Ha75}: if the Euclidean distance $d(K,E)\geq
%1$,
%\[
%\w\bigl(\conv(K\cup E)\bigr)\geq \w(K)+\w(E).
%\]
\item Hadwiger \cite[(2.4)]{Ha75}: for all $v\in\s^{n-1}$ and any $r\geq 0$,
\[
\w(K)\geq\frac{\w(K\cap H_{v,0})+\w(K\cap H_{v,r})}{2}+\int_0^r\w(K\cap
H_{v,t})\,\dlat t.
\]
\item Wills \cite[(4.4)]{W75}: the $i$-th derivative of $\w(-\lambda
B^n_2)$ satisfies
\[
\frac{d^i\w(-\lambda
B^n_2)}{d\lambda^i}=\frac{n!\,\kappa_n}{i!\,\kappa_i}\w(-\lambda
B^i_2).
\]
\item McMullen \cite[Theorem 2]{Mc91}:
\[
\w(K)\leq e^{\V_1(K)}.
\]
\end{enumerate}
\end{theo}

Moreover, in \cite{Ha75} Hadwiger also showed several integral
representations of $\w(K)$. We emphasize the following ones (see
\cite[(1.3) and (1.4)]{Ha75}):
\begin{equation}\label{e:representacionesWills}
\begin{split}
\w(K) & =\int_{\R^n}e^{-\pi d(x,K)^2}\dlat x,\\
 \w(K) & =2\pi\int_0^{\infty}\vol(K+tB^n_2)t\,e^{-\pi t^2}\dlat t.
\end{split}
\end{equation}

In this paper we obtain new properties of the Wills functional (indeed, we
will work in a more general setting that will be stated in
Section~\ref{s:basic}).

The paper is organized as follows. In Section \ref{s:basic} we motivate
and introduce, on one hand, an extension of the classical Wills functional
in a more general setting, as well as further notation and basic results
for functions that will be needed later on. On the other hand, we prove
some preliminary results as, for instance, the useful property that the
Wills functional can be seen as the Asplund product of two log-concave
functions (Lemma~\ref{lem:Asplund}), which will be a key result throughout
the paper.

Next, Section \ref{s:bounds_W} is devoted to providing upper and lower
bounds for the classical Wills functional of a convex body $K$ in terms of
other functionals. For instance, $\w(K)$ can be bounded by the volume of
$K$ as follows:
\begin{theorem}\label{t:8^n/2vol}
Let $K\in\K^n$ be a convex body with non-empty interior. Then
\begin{equation*}
\left(8^{n/2}\vol(K)\right)^{1/2}\leq\w(K)
\leq\frac{8^{n/2}\vol(K)}{\sup_{y\in\R^n}\gamma_n\bigl(y+\sqrt{2\pi}K\bigr)}.
\end{equation*}
\end{theorem}
Here $\gamma_n$ denotes the standard Gaussian probability measure on
$\R^n$, this~is,
\begin{equation}\label{e:gaussian_measure}
\dlat\gamma_n(x)=\frac{1}{(2\pi)^{n/2}}e^{-\frac{|x|^2}{2}}\,\dlat x.
\end{equation}
We will also reprove McMullen's result (see
Theorem~\ref{t:properties_Wills} iv)) using a different approach, and will
obtain a lower bound for the Wills functional in the same spirit as
McMullen's bound. To this end we denote by $\cir(K)$ the {\em
circumradius} of $K$, i.e., $\cir(K)=\min\{R>0:\exists\,x\in\R^n\text{
with }K\subset x+RB_2^n\}$.
\begin{theorem}\label{t:McMullen}
Let $K\in\K^n$. Then
\[
e^{\V_1(K)-\pi\cir(K)^2}\leq\w(K)\leq e^{\V_1(K)}.
\]
\end{theorem}

Next, we show that the cube $[-1,1]^n$ maximizes the Wills functional
among all $0$-symmetric convex bodies in John position. We recall that a
convex body $K$ is said to {\it be in John position} if the maximum volume
ellipsoid contained in $K$ is the Euclidean unit ball.
\begin{theorem}\label{t:John}
Let $K\in\K^n$ be a $0$-symmetric set in John position. Then
\begin{equation*}
\w(K)\leq\w\bigl([-1,1]^n\bigr).
\end{equation*}
\end{theorem}

In Section \ref{s:Wills_B-M} we study some Brunn-Minkowski type
inequalities for the Wills functional. As we will see, $\w(\cdot)$ is not,
in general, ($1/n$)-concave and so, either less restrictive concavity or
additional constants are needed. In this regard we show, among others, the
following results:
\begin{theorem}\label{t:BM_mult_W}
Let $K,L\in\K^n$ and let $\lambda\in(0,1)$. Then
\begin{equation*}
\w\bigl((1-\lambda)K+\lambda L\bigr)\geq\w(K)^{1-\lambda}\w(L)^\lambda.
\end{equation*}
\end{theorem}

\begin{theorem}\label{t:BM_W_constant}
Let $K,L\in\K^n$ and let $\lambda\in(0,1)$. Then
\begin{equation*}
\w\bigl((1-\lambda)K+\lambda L\bigr)^{1/n}
\geq\frac{1}{(n!)^{1/n}}\left((1-\lambda)\w(K)^{1/n}
+\lambda\w(L)^{1/n}\right).
\end{equation*}
\end{theorem}

However, although $\w(\cdot)$ does not satisfy a Brunn-Minkowski
inequality in its classical form (i.e., with exponent $1/n$), a reverse
Brunn-Minkowski inequality, analogous to the one for the volume proved by
Milman \cite{Mi}, holds for the Wills functional:
\begin{theorem}\label{t:reverse_BM_W}
Let $K,L\in\K^n$. Then there exist $T\in\SL(n)$ and an absolute constant
$C>0$ such that
\begin{equation*}
\w(K+TL)^{1/n}\leq C\left[\w(K)^{1/n}+\w(L)^{1/n}\right].
\end{equation*}
\end{theorem}

The last section of the paper is devoted to studying Rogers-Shephard type
inequalities for the classical Wills functional. Among others, we get a
section/projection Rogers-Shephard type relation for $\w(\cdot)$, as well
as an upper bound for the Wills functional of the difference body
$K-K:=K+(-K)$.

\begin{theorem}\label{t:R-S_Wills}
Let $K\in\K^n$ with $0\in K$ and let $H\in\G(n,k)$. Then
\[
\w(P_HK)\w(K\cap H^{\bot})
\leq\min\left\{\binom{n}{k}\w(K),2^{n/2}\,\w\bigl(\sqrt{2}K\bigr)\right\}.
\]
\end{theorem}

\begin{theorem}\label{t:R-S_Wills_K,L}
Let $K, L\in\K^n$. Then
\[
\w\left(K\cap L\right)\w\left(\frac{K-L}{2}\right)\leq 2^n\,\w(K)\w(L).
\]
In particular we get
\begin{equation}\label{e:R-S_K-K1}
\w(K-K)\leq 2^n\,\w(2K).
\end{equation}
\end{theorem}

\section{The generalized Wills functional. Some preliminary results}\label{s:basic}

The integral expressions \eqref{e:representacionesWills} showed by
Hadwiger in \cite{Ha75} have turned out crucial in many respects.
Recently, Kampf \cite{K} proved certain generalizations of them when the
`distance' $d_E(x,K)$, between $x\in\R^n$ and $K$, relative to a convex
body $E$ with $0\in\inter E$, is considered, i.e., for
\begin{equation}\label{e:d_E(x,K)}
d_E(x,K)=\min_{y\in K}|x-y|_E=\min\{t\geq 0:x\in K+tE\}
\end{equation}
(see also \cite{HCY15} for the more general case in which the assumption
$0\in\inter E$ is not required). He showed that
\begin{equation}\label{e:representacionesWills_E}
\int_{\R^n}\!\!e^{-\pi d_E(x,K)^2}\dlat
x=2\pi\!\int_0^{\infty}\!\!\vol(K+tE)\,te^{-\pi t^2}\dlat t
=\!\sum_{i=0}^n\!\binom{n}{i}\frac{\W_i(K;E)}{\kappa_i}.
\end{equation}
We observe that $f(x)=e^{-\pi d_E(x,K)^2}$ is a {\it log-concave function}
because the distance function $d_E(\,\cdot,K)$ is convex (cf. e.g.
\cite[Lemma 1.5.9]{Sch}); we recall that $f:\R^n\longrightarrow\R_{\geq
0}$ is said to be log-concave if
\[
f\bigl((1-\lambda)x+\lambda y\bigr)\geq
f(x)^{1-\lambda}f(y)^{\lambda}\quad \text{ for any }\lambda\in(0,1)
\;\text{ and all }\; x,y\in\R^n,
\]
or equivalently, if it is of the form
\[
f(x)=e^{-u(x)}
\]
for $u:\R^n\longrightarrow\R\cup\{\infty\}$ a convex function.

The study of log-concave functions has become very important in
the recent years, among others in the study of problems related to
the distribution of mass in a convex body (we refer the reader,
for instance, to \cite{G,KM} and the references therein). Notice
also that convex bodies are contained inside the class of
log-concave functions via either the negative exponential of the
Minkowski functional $|\,\cdot|_K$ of a convex body $K$ (whose
integral is, up to a constant, the volume), or its characteristic
function, which we will denote throughout the paper by
$\chi_{_K}$, i.e.,
\[
\chi_{_K}(x)=\left\{
\begin{array}{ll}
1 & \text{ if }x\in K,\\
0 & \text{ otherwise.}
\end{array}
\right.
\]

Going back to the integral representation
\eqref{e:representacionesWills_E}, we observe that a more general
functional can be obtained by replacing $e^{-\pi t^2}$ by another function
$G(t)$ properly associated to a log-concave measure $\mu$ on the
non-negative real line $\R_{\geq0}$. Thus, let $\mu$ be the measure on
$\R_{\geq 0}$ given by $\dlat\mu(t)=\phi(t)\,\dlat t$, where
$\phi:\R_{\geq 0}\longrightarrow\R_{\geq 0}$ is log-concave. Then, the
associated function $G(t)=\mu\bigl([t,\infty)\bigr)$ is decreasing and
log-concave, and hence $G(t)=e^{-u(t)}$ for a convex function $u:\R_{\geq
0}\longrightarrow\R\cup\{\infty\}$ which is increasing and, without loss
of generality, can be assumed to be continuous in $\bigl\{t\in\R_{\geq
0}:u(t)\neq\infty\bigr\}$. Moreover, if $\mu$ is a probability measure,
then $G(0)=1$, i.e., $u(0)=0$.

Thus, we may consider the more general expression
\begin{equation*}\label{e:f(KEmu)}
\f{K}{E}{u}(x)=G\bigl(d_E(x,K)\bigr)=e^{-u\bigl(d_E(x,K)\bigr)} \quad
\text{ for }\; K,E\in\K^n,
\end{equation*}
where $u:\R_{\geq 0}\longrightarrow\R\cup\{\infty\}$ satisfies the above
assumptions (monotonicity, continuity and convexity). From now on, and for
the sake of brevity, we will denote by $\co(\R_{\geq 0})$ the family of
all convex functions $u:\R_{\geq 0}\longrightarrow\R\cup\{\infty\}$ which
are continuous in $\bigl\{t\in\R_{\geq 0}:u(t)\neq\infty\bigr\}$ and
increasing.

The desired outcome that integration of our function $\f{K}{E}{u}$
provides a Steiner type formula (with weights) indeed holds (see also
\cite[Proposition~3]{K} and \cite[Lemma~1.1]{HCY15}).

\begin{proposition}\label{lem:ComputationIntegral}
Let $K,E\in\K^n$ with $0\in\inter E$ and let $u\in\co(\R_{\geq 0})$ be
strictly increasing. Then
\begin{equation}\label{e:ComputationIntegral}
\int_{\R^n}\f{K}{E}{u}(x)\,\dlat x
=\int_{u(0)}^\infty\vol\bigl(K+u^{-1}(s)E\bigr)e^{-s}\,\dlat s
=\sum_{i=0}^n\binom{n}{i}\m_i^u\,\W_i(K;E),
\end{equation}
where
\[
\m_i^u=\displaystyle\int_{u(0)}^{\infty}u^{-1}(s)^ie^{-s}\,\dlat s.
\]
\end{proposition}
We notice that when dealing with a differentiable function $u$, the
numbers $\m_i^u$ are nothing else but the {\it moments} of the measure
$\mu$ associated to $u$, i.e., such that
$G(t)=\mu\bigl([t,\infty)\bigr)=e^{-u(t)}$:
\begin{equation}\label{e:m_i^u}
\m_i^u=\int_{u(0)}^{\infty}u^{-1}(s)^ie^{-s}\,\dlat s=\int_{0}^\infty
t^i\,\dlat\mu(t).
\end{equation}

\begin{proof}[Proof of Proposition \ref{lem:ComputationIntegral}]
Using Fubini's theorem we get
\[
\begin{split}
\int_{\R^n}\f{K}{E}{u}(x)\,\dlat x  &
    =\int_{\R^n}e^{-u\bigl(d_E(x,K)\bigr)}\dlat x
    =\int_{\R^n}\int_{u(d_E(x,K))}^{\infty}e^{-s}\,\dlat s\,\dlat x\\
 & =\int_{u(0)}^{\infty}e^{-s}\int_{\R^n}\chi_{\{y\in\R^n:u(d_E(y,K))\leq s\}}(x)\,\dlat x\,\dlat s.
\end{split}
\]
Now, the strict monotonicity of $u(t)$ ensures the existence of the
inverse $u^{-1}(s)$, and hence, by \eqref{e:d_E(x,K)},
\[
\begin{split}
\Bigl\{x\in\R^n:u\bigl(d_E(x,K)\bigr)\leq s\Bigr\} &
    =\bigl\{x\in\R^n:d_E(x,K)\leq u^{-1}(s)\bigr\}\\
 & =K+u^{-1}(s)E.
\end{split}
\]
Thus, via the Steiner formula \eqref{e:Steinerform} we can conclude that
\[
\begin{split}
\int_{\R^n}\f{K}{E}{u}(x)\,\dlat x  &
    =\int_{u(0)}^{\infty}e^{-s}\int_{\R^n}\chi_{\{y\in\R^n:u(d_E(y,K))\leq s\}}(x)\,\dlat x\,\dlat s\\
 & =\int_{u(0)}^{\infty}\vol\Bigl(K+u^{-1}(s)E\Bigr)e^{-s}\,\dlat s\\
 & =\int_{u(0)}^{\infty}\sum_{i=0}^n\binom{n}{i}\W_i(K;E)u^{-1}(s)^{i}e^{-s}\,\dlat s\\
 & =\sum_{i=0}^n\left[\binom{n}{i}\W_i(K;E)\int_{u(0)}^{\infty}u^{-1}(s)^{i}e^{-s}\,\dlat s\right].\qedhere
\end{split}
\]
\end{proof}

\begin{example}
(i) The classical Wills functional $\w(K)$ is obtained for
$\f{K}{B_2^n}{\u}$ (cf. \eqref{e:FWills}) because, clearly, the moments
\begin{equation}\label{e:m_i_u_0}
\m_i^{\u}=\frac{1}{\kappa_i}.
\end{equation}

\noindent (ii) In the case of the function
$u_p(t)=\bigl(2\Gamma(1+1/p)t\bigr)^p$, $p>1$, we have
\[
\int_{u_p(0)}^{\infty}u_p^{-1}(s)^{i}e^{-s}\,\dlat s
    =\int_0^{\infty}\frac{s^{i/p}}{2^i\Gamma\bigl(1+\tfrac{1}{p}\bigr)^i}\,e^{-s}\,\dlat s
    =\frac{\Gamma\left(1+\frac{i}{p}\right)}{\left(2\Gamma\bigl(1+\frac{1}{p}\bigr)\right)^i}=\frac{1}{\vol_i(B_p^i)},
\]
where $B^i_p$ is the unit $p$-ball associated to the $p$-norm
$|x|_p=\bigl(\sum_{i=1}^n|x_i|^p\bigr)^{1/p}$. So, we~get
\[
\int_{\R^n}\f{K}{E}{u_p}(x)\,\dlat
x=\sum_{i=0}^n\binom{n}{i}\frac{1}{\vol_i(B_p^i)}\W_i(K;E).
\]
\end{example}

As mentioned in the introduction, $\w(K)$ depends only on the convex body,
but not on the dimension of the embedding space. However, in the most
general case of the function $\f{K}{E}{u}$, it does not come true anymore.
Therefore, we need a special notation which allows us to distinguish the
dimension in which the Wills functional is computed, as well as the
involved sets. Thus, if $K,E\in\K^n$ with $\dim K=\dim E=k$ and we compute
the extended Wills functional in $\R^k$ (assuming that both $K,E$ lie in
the same affine $k$-dimensional space), we will write
\[
\wk{k}{u}(K;E):=\int_{\R^k}\f{K}{E}{u}(x)\,\dlat x,
\]
whereas the $n$-dimensional functional will be denoted just by
$\w_{u}(K;E)$. Moreover, we will keep the classical notation for the usual
Wills functional, i.e., $\w(K):=\w_{\u}(K;B_2^n)$.

%As a first direct consequence of Proposition \ref{lem:ComputationIntegral}
%we get that the average Wills functional of the projections of the set
%over the Grassmannian $\G(n,k)$ is a truncated part of the full Wills
%functional (up to a constant).
%
%\begin{corollary}
%Let $K,E\in\K^n$ with $0\in\inter E$, $u\in\co(\R_{\geq 0})$ strictly
%increasing, and $H\in\G(n,k)$. Then
%\[
%\frac{\kappa_n}{\kappa_k}\int_{\G(n,k)}\wk{k}{u}(P_HK;B^k_2)\,\dlat
%H=\sum_{i=0}^k\binom{k}{i}\m_i^u\,\W_{n-k+i}(K),
%\]
%where integration is taken with respect to the rotation-invariant
%probability measure on $\G(n,k)$.
%\end{corollary}
%
%\begin{proof}
%Kubota's integral recursion formula (see \cite[(5.72)]{Sch}),
%establishes that
%\begin{equation}\label{eq:kubota}
%\W_{n-k+i}(K)=\frac{\kappa_n}{\kappa_k}\int_{G(n,k)}\Wk{k}{i}(P_HK)\,\dlat
%H\quad i=0,\dots,k.
%\end{equation}
%Then, from Proposition \ref{lem:ComputationIntegral} we get
%\[
%\begin{split}
%\int_{\G(n,k)}\wk{k}{u}(P_HK;B^k_2)\,\dlat H &
%    =\int_{\G(n,k)}\int_{H}\f{P_HK}{B^k_2}{u}(x)\,\dlat x\,\dlat H\\
% & =\sum_{i=0}^k\binom{k}{i}\m_i^u\int_{\G(n,k)}\Wk{k}{i}(P_HK)\,\dlat H\\
% & =\frac{\kappa_k}{\kappa_n}\sum_{i=0}^k\binom{k}{i}\m_i^u\,\W_{n-k+i}(K).\qedhere
%\end{split}
%\]
%\end{proof}

\subsection{Some basics on (log-concave) functions}

Since log-concave functions are the keystone of the coming development, we
need to recall some background.

A first major result for (measurable) functions is the so-called {\it
Borell-Brascamp-Lieb inequality} (see e.g. \cite[Theorem~10.1]{G}):

\begin{theo}[Borell-Brascamp-Lieb's inequality]\label{t:BBL}
Let $\lambda\in(0,1)$ be fixed, let $-1/n\leq p\leq\infty$ and let
$f,g,h:\R^n\longrightarrow\R_{\geq0}$ be non-negative measurable functions
such that, for any $x,y\in\R^n$ with $f(x)g(y)>0$,
\begin{equation*}
h\bigl((1-\lambda) x+\lambda y\bigr)\geq \bigl((1-\lambda)f(x)^p+\lambda
g(y)^p\bigr)^{1/p}.
\end{equation*}
Then
\begin{equation*}\label{e:BBL}
\int_{\R^n}h(x)\,\dlat x\geq \left[(1-\lambda)\left(\int_{\R^n}f(x)\,\dlat
x\right)^q+\lambda\left(\int_{\R^n}g(x)\,\dlat x\right)^q \right]^{1/q},
\end{equation*}
provided that $f$ and $g$ have non-zero integrals, and where $q=p/(np+1)$.
\end{theo}
The case $p=0$ of the Borell-Brascamp-Lieb inequality is known as the {\it
Pr\'ekopa-Leindler inequality}: if $h\bigl((1-\lambda) x+\lambda
y\bigr)\geq f(x)^{1-\lambda}g(y)^{\lambda}$ then
\begin{equation}\label{e:PrekopaLeindler}
\int_{\R^n}h(x)\,\dlat x\geq\left(\int_{\R^n}f(x)\,\dlat
x\right)^{1-\lambda}\left(\int_{\R^n}g(x)\,\dlat x\right)^{\lambda}.
\end{equation}

For $p\geq 1$, we recall that the {\it $p$-norm} of a function
$f:\R^n\longrightarrow\R_{\geq 0}$ is defined as
\[
\|f\|_p=\left(\int_{\R^n}f(x)^p\dlat x\right)^{1/p},
\]
where the case $p=\infty$ has to be understood as
\[
\|f\|_{\infty}=\sup_{x\in\R^n}f(x).
\]
Next we set the definition of three important operations for functions:
convolution, Asplund product and difference function.

\begin{definition}
The {\em Asplund product} of two functions
$f,g:\R^n\longrightarrow\R_{\geq 0}$ is given by
\[
(f\star g)(z)=\sup_{z=x+y}f(x)g(y),
\]
and their {\em convolution} is defined as
\[
(f*g)(z)=\int_{\R^n}f(x)g(z-x)\,\dlat x.
\]
For $\lambda\in(0,1)$, the {\em $\lambda$-difference function} associated
with $f$ and $g$ is given by
\[
\Delta_{\lambda}^{f,g}(z)=\sup_{z=(1-\lambda)x+\lambda y}
f\left(\frac{x}{1-\lambda}\right)^{1-\lambda}g\left(\frac{-y}{\lambda}\right)^{\lambda}.
\]
\end{definition}
Notice that in the case of characteristic functions of two convex bodies
$K,L\in\K^n$, the convolution
$\chi_{_K}*\chi_{_L}=\vol\bigl(K\cap(\,\cdot-L)\bigr)$ whereas the Asplund
product $\chi_{_K}\star\chi_{_E}=\chi_{_{K+E}}$, and so it plays the role
of the Minkowski addition of convex bodies in the setting of log-concave
functions. In the last decades, many efforts have been made in order to
extend geometric inequalities for convex bodies to the more general
setting of log-concave functions.

In \cite[Theorem~1.8]{AEFO} it is proved that if
$f,g:\R^n\longrightarrow\R_{\geq 0}$ are log-concave functions and
$\lambda\in(0,1)$, then the following Rogers-Shephard type inequality
holds:
\begin{equation}\label{e:diff_function}
\int_{\R^n}f(x)^{\lambda}g(x)^{1-\lambda}\,\dlat x
\int_{\R^n}\Delta_{\lambda}^{f,g}(z)\,\dlat z \leq\int_{\R^n}f(x)\,\dlat
x\int_{\R^n}g(y)\,\dlat y.
\end{equation}
Moreover (see \cite[Theorem~2.1]{AGJV}),
\begin{equation}\label{eq:RSfunctions}
\|f*g\|_\infty\int_{\R^n}(f\star g)(x)\,\dlat
x\leq\binom{2n}{n}\|f\|_\infty \|g\|_\infty\int_{\R^n}f(x)\,\dlat
x\int_{\R^n}g(x)\,\dlat x.
\end{equation}

Another outstanding result for log-concave functions, providing a reverse
Brunn-Minkowski inequality, was obtained by Klartag and Milman in
\cite[Theorem~1.3 and Remark (2) in page~181]{KM}:
\begin{theo}\label{t:Klartag_Milman}
Let $f,g:\R^n\longrightarrow\R_{\geq 0}$ be log-concave functions
with finite and positive integrals, such that
$\|f\|_{\infty}=\|g\|_{\infty}=f(0)=g(0)=1$. Then there exist
$T_1,T_2\in\SL(n)$ and an absolute constant $C>0$ such that
\[
\begin{split}
\biggl(\int_{\R^n}\bigl((f\circ T_1) & \star(g\circ T_2)\bigr)(x)\,\dlat x\biggr)^{1/n}\\
 & \leq C\left[\left(\int_{\R^n}(f\circ T_1)(x)\,\dlat x\right)^{1/n}
    +\left(\int_{\R^n}(g\circ T_2)(x)\,\dlat x\right)^{1/n}\right].
\end{split}
\]
\end{theo}

We conclude this subsection recalling one last concept: the projection of
a function (see e.g. \cite{KM}).

\begin{definition}
Let $f:\R^n\longrightarrow\R_{\geq 0}$ and let $H\in\G(n,k)$. The
\emph{projection} of $f$ onto $H$ is the function
$P_Hf:H\longrightarrow\R_{\geq0}\cup\{\infty\}$ defined by
\begin{equation*}
(P_Hf)(x)=\sup_{y\in H^{\bot}}f(x+y).
\end{equation*}
\end{definition}

The geometric meaning of this definition is easy: the (strict) hypograph
of $P_Hf$ is the projection of the (strict) hypograph of $f$ onto $H$. In
particular, for $K\in\K^n$, $P_H\chi_{_K}=\chi_{_{P_HK}}$.

Regarding the projection of a function $f:\R^n\longrightarrow\R_{\geq 0}$,
very recently \cite[Theorems~1.1 and~1.2]{AA} the following
Rogers-Shephard type inequalities for log-concave (integrable) functions
have been obtained: for $H\in\G(n,k)$,
\begin{equation}\label{e:R-S_f_log-concave1}
\int_H(P_Hf)(x)\,\dlat x\int_{H^{\bot}}f(y)\,\dlat y
\leq\binom{n}{k}\|f\|_{\infty}\int_{\R^n}f(z)\,\dlat z;
\end{equation}
moreover, if $\|f\|_{\infty}=f(0)$ then, for any $\lambda\in(0,1)$,
\begin{equation}\label{e:R-S_f_log-concave2}
(1-\lambda)^k\lambda^{n-k}\int_H(P_Hf)(x)^{1-\lambda}\dlat x
\int_{H^{\bot}}f(y)^{\lambda}\dlat y \leq\int_{\R^n}f(z)\,\dlat z.
\end{equation}

\subsection{Some preliminary lemmas}

We start this subsection showing several properties for the
Asplund product that will be needed later on.

\begin{lemma}\label{l:asplund_fg}
Let $f,g:\R^n\longrightarrow\R_{\geq 0}$ be two non-negative functions,
and let $T_1,T_2\in\SL(n)$. Then, for all $x\in\R^n$,
\[
\bigl((f\circ T_1)\star(g\circ T_2)\bigr)(x)=\Bigl(f\star
\bigl(g\circ(T_2T_1^{-1})\bigr)\Bigr)\bigl(T_1(x)\bigr).
\]
\end{lemma}

\begin{proof}
With the changes of variable $w=T_1y$ and $z=T_1x$, we get
\[
\begin{split}
\bigl((f\circ T_1)\star(g\circ T_2)\bigr)(x) &
    =\sup_{y\in\R^n}f(T_1y)g\bigl(T_2(x-y)\bigr)\\
 & =\sup_{w\in\R^n}f(w)g\bigl(T_2x-T_2T_1^{-1}w\bigr)\\
 & =\sup_{w\in\R^n}f(w)g\bigl(T_2T_1^{-1}(z-w)\bigr)\\
 &=\Bigl(f\star\bigl(g\circ(T_2T_1^{-1})\bigr)\Bigr)\bigl(T_1(x)\bigr).\qedhere
\end{split}
\]
\end{proof}

The function $\f{K}{E}{u}$ defining the generalized Wills functional can
be seen as the Asplund product of two log-concave functions. This property
will be helpful in the subsequent results.

\begin{lemma}\label{lem:Asplund}
Let $u\in\co(\R_{\geq 0})$ and let $K,E\in\K^n$ with $0\in\inter E$. Then
\[
\f{K}{E}{u}=e^{-u\bigl(|\,\cdot\,|_E\bigr)}\star\chi_{_K}.
\]
\end{lemma}

\begin{proof}
By definition of Asplund product and since $u$ is increasing,
\[
\begin{split}
\left(e^{-u(|\,\cdot\,|_E)}\star\chi_{_K}\right)(x) &
    =\sup_{y\in\R^n}e^{-u\bigl(|x-y|_E\bigr)}\chi_{_K}(y)\\
 & =\sup_{y\in K}e^{-u\bigl(|x-y|_E\bigr)}=e^{-\inf_{y\in K}u\bigl(|x-y|_E\bigr)}\\
 & =e^{-u\bigl(\inf_{y\in K}|x-y|_E\bigr)}=e^{-u\bigl(d_E(x,K)\bigr)}
\end{split}
\]
for all $x\in\R^n$.
\end{proof}

Also the projection of a function ``behaves well'' with our function
$\f{K}{E}{u}$:

\begin{lemma}\label{l:P_Hf_K=f_P_HK}
Let $K,E\in\K^n$ with $0\in\inter E$, $u\in\co(\R_{\geq 0})$ and
$H\in\G(n,k)$. Then, for every $x\in H$,
\[
\left(P_H\f{K}{E}{u}\right)(x)=\f{P_HK}{P_HE}{u}(x).
\]
Moreover, $\f{K}{E}{u}(x)\geq\f{K\cap H}{E\cap H}{u}(x)$.
\end{lemma}

\begin{proof}
If we show that
\begin{equation}\label{e:d(projection)}
d_{P_HE}(x,P_HK)=\inf_{y\in H^{\bot}}d_E(x+y,K),
\end{equation}
then we immediately get the first equality, because in such case
\[
\begin{split}
\left(P_H\f{K}{E}{u}\right)(x) &
    =\sup_{y\in H^{\bot}}\f{K}{E}{u}(x+y)=\sup_{y\in H^{\bot}}e^{-u\bigl(d_E(x+y,K)\bigr)}\\
 & =e^{-\inf_{y\in H^{\bot}}u\bigl(d_E(x+y,K)\bigr)}
    =e^{-u\bigl(\inf_{y\in H^{\bot}}d_E(x+y,K)\bigr)}\\
 & =e^{-u\bigl(d_{P_HE}(x,P_HK)\bigr)}=\f{P_HK}{P_HE}{u}(x).
\end{split}
\]
So, we have to prove \eqref{e:d(projection)}. Let $x\in H$ and $y\in
H^{\bot}$. If $x+y\in K+tE$ then $x=P_H(x+y)\in P_H(K+tE)=P_HK+tP_HE$, and
hence
\[
\{t\geq 0:x+y\in K+tE\}\subset\{t\geq 0:x\in P_HK+tP_HE\}.
\]
Therefore, from the definition of $d_E$ (see \eqref{e:d_E(x,K)}) we
conclude that
\begin{equation}\label{e:d_E_d_PHE}
d_E(x+y,K)\geq d_{P_HE}(x,P_HK)\quad\text{ for all }\; y\in H^{\bot}.
\end{equation}
Now, let $t_0=d_{P_HE}(x,P_HK)$. Then $x\in P_HK+t_0P_HE=P_H(K+t_0E)$ and
so there exists $y_0\in H^{\bot}$ such that $x+y_0\in K+t_0E$. This
implies that $t_0\geq d_E(x+y_0,K)$ which, together with
\eqref{e:d_E_d_PHE}, yields $t_0=d_E(x+y_0,K)$. Then we can conclude that
\[
\inf_{y\in H^{\bot}}d_E(x+y,K)=t_0=d_{P_HE}(x,P_HK),
\]
as required.

The proof of the second assertion of the lemma is straightforward: since
$(K\cap H)+t(E\cap H)\subset K+tE$, \eqref{e:d_E(x,K)} yields $d_{E\cap
H}(x,K\cap H)\geq d_E(x,K)$, and so
\[
\f{K}{E}{u}(x)=e^{-u\bigl(d_E(x,K)\bigr)}\geq e^{-u\bigl(d_{E\cap
H}(x,K\cap H)\bigr)}=\f{K\cap H}{E\cap H}{u}(x).\qedhere
\]
\end{proof}

Finally, we compute the $p$-norm of the function $\f{K}{E}{u}$. Given
$K,E\in\K^n$ with $0\in\inter E$, $u\in\co(\R_{\geq 0})$ and $1\leq
p<\infty$, we clearly have
\begin{equation}\label{e:|f_u|_p}
\|\f{K}{E}{u}\|_p=\left(\int_{\R^n}e^{-pu\bigl(d_E(x,K)\bigr)}\,\dlat
x\right)^{1/p}=\w_{pu}(K;E)^{1/p}.
\end{equation}
In particular, for $u=\u$ and $E=B^n_2$, we get the following relation.

\begin{lemma}\label{l:p-norm}
Let $K\in\K^n$ and $1\leq p<\infty$. Then
\begin{equation}\label{e:f_n/n-1}
\left\|\f{K}{B^n_2}{\u}\right\|_p=\frac{1}{p^{n/(2p)}}\w\bigl(\sqrt{p}\,K\bigr)^{1/p}.
\end{equation}
\end{lemma}

\begin{proof}
Doing the change of variable $\sqrt{p}\,x=y$, we immediately get
\[
\begin{split}
\left\|\f{K}{B^n_2}{\u}\right\|_p & =\left(\int_{\R^n}e^{-p\pi
d(x,K)^2}\,\dlat x\right)^{1/p}
    =\left(\int_{\R^n}e^{-\pi d\bigl(\sqrt{p}\,x,\sqrt{p}\,K\bigr)^2}\,\dlat x\right)^{1/p}\\
 & =\left(\frac{1}{p^{n/2}}\int_{\R^n}e^{-\pi d\bigl(y,\sqrt{p}\,K\bigr)^2}\,\dlat y\right)^{1/p}
    =\frac{1}{p^{n/(2p)}}\w\bigl(\sqrt{p}\,K\bigr)^{1/p}.\qedhere
\end{split}
\]
\end{proof}

\section{Bounding the Wills functional of a convex body}\label{s:bounds_W}

In this section we provide upper and lower bounds for the Wills functional
of a convex body $K$ in terms of other functionals. Our first aim will be
to relate $\w_u(K;E)$ with the volumes of the involved sets.

\subsection{Relating the Wills functional to the volume}

We start this subsection by showing Theorem \ref{t:8^n/2vol}, which will
be a consequence of a more general result for $\w_u(K;E)$. First we study
the lower bound.

\begin{theorem}
Let $K,E\in\K^n$ with $0\in\inter E$ and let $u\in\co(\R_{\geq 0})$. Then,
for any $\lambda\in(0,1)$,
\begin{equation}\label{e:lowerboundWills_u_E}
\w_{u}(K;E)\geq\left(\frac{1}{\lambda^{\lambda}(1-\lambda)^{1-\lambda}}\right)^n
\left(\int_{\R^n}e^{-\frac{u(|x|_E)}{1-\lambda}}\,\dlat
x\right)^{1-\lambda}\vol(K)^\lambda.
\end{equation}
In particular, for $u=\u$ and $E=B^n_2$, we have
\begin{equation}\label{e:lowerboundWills_classical}
\w(K)\geq\left(\frac{1}{\lambda^{\lambda}(1-\lambda)^{(1-\lambda)/2}}\right)^n\vol(K)^\lambda.
\end{equation}
\end{theorem}

\begin{proof}
We consider the functions
\[
f=e^{-u\bigl((1-\lambda)|\,\cdot\,|_E\bigr)/(1-\lambda)},\quad
g=\chi_{_{K/\lambda}}\quad\text{and}\quad
h=f\left(\frac{\cdot}{1-\lambda}\right)^{1-\lambda}\star
g\left(\frac{\cdot}{\lambda}\right)^{\lambda},
\]
for which $h=e^{-u(|\,\cdot\,|_E)}\star\chi_{_K}=\f{K}{E}{u}$ by
Lemma \ref{lem:Asplund}. Then, the Pr\'ekopa-Leindler inequality
(Theorem \ref{t:BBL} for $p=0$) applied to $f$, $g$ and $h$
yields~\eqref{e:lowerboundWills_u_E}.

When $E=B^n_2$ and $u=\u$, and since the standard Gaussian measure
$\gamma_n$ is a probability measure, we get (cf.
\eqref{e:gaussian_measure})
\begin{equation}\label{e:int_exp}
\int_{\R^n}e^{-\frac{\pi}{1-\lambda}|x|^2}\,\dlat x
=\frac{(1-\lambda)^{n/2}}{(2\pi)^{n/2}}\int_{\R^n}e^{-\frac{|x|^2}{2}}\,\dlat
x=(1-\lambda)^{n/2}.
\end{equation}
This finishes the proof.
\end{proof}

Following the idea of the above proof but now using the linear refinement
of the Pr\'ekopa-Leindler inequality (see \cite[Theorem~1.5]{CSY}), we
immediately get the following improvement of
\eqref{e:lowerboundWills_classical} for the classical Wills functional.

\begin{corollary}
Let $K\in\K^n$ be such that there exists $H\in\G(n,n-1)$ with
\[
\vol_{n-1}(P_HK)=\left(\frac{\lambda}{\sqrt{1-\lambda}}\right)^{n-1}
\]
for some $\lambda\in(0,1)$. Then
\[
\w(K)\geq\frac{\vol(K)}{\lambda^{n-1}}+\frac{1}{(1-\lambda)^{(n-2)/2}}.
\]
\end{corollary}

\begin{proof}
Given two log-concave functions
$f,g:\R^n\longrightarrow\R_{\geq0}$ (decaying to zero at
infinity), if there exists a hyperplane $H\in\G(n,n-1)$ such that
$P_Hf$ and $P_Hg$ have the same (finite) integral, then the
right-hand side in the Pr\'ekopa-Leindler
inequality~\eqref{e:PrekopaLeindler} (Theorem \ref{t:BBL} for
$p=0$) can be replaced by the arithmetic mean of the integrals of
$f$ and $g$ (see \cite[Theorem~1.5]{CSY}):
\begin{equation}\label{e:PL_lin}
\int_{\R^n}h(x)\,\dlat x\geq (1-\lambda) \int_{\R^n}f(x)\,\dlat x
+\lambda\int_{\R^n}g(x)\,\dlat x.
\end{equation}

Let $L=\bigl(\sqrt{1-\lambda}/\lambda\bigr)K$,
$f=e^{-\pi|\,\cdot\,|^2}$, $g=\chi_{_{L}}$ and
$h=f\bigl(\frac{\cdot}{1-\lambda}\bigr)^{1-\lambda}\star
g\bigl(\frac{\cdot}{\lambda}\bigr)^{\lambda}$. Since
\[
(P_Hf)(x)=\sup_{y\in H^{\bot}}e^{-\pi|x+y|^2}=e^{-\pi|x|^2}\sup_{y\in
H^{\bot}}e^{-\pi|y|^2}=e^{-\pi|x|^2}
\]
for all $x\in H$, then (cf.\eqref{e:int_exp})
\[
\int_H(P_Hf)(x)\,\dlat x=1=\vol_{n-1}(P_HL)=\int_H(P_Hg)(x)\,\dlat x,
\]
and hence we are in the above conditions and we can apply the
mentioned refinement of the Pr\'ekopa-Leindler inequality. Thus,
using Lemma~\ref{lem:Asplund},
\[
h=e^{-\frac{\pi}{1-\lambda}|\,\cdot\,|^2}\star\chi_{_{\lambda L}}
=e^{-\frac{\pi}{1-\lambda}d(\,\cdot\, ,\lambda L)^2} =e^{-\pi
d\bigl(\frac{\cdot}{\sqrt{1-\lambda}},K\bigr)^2},
\]
and so, integrating $f$, $g$ and $h$, doing the change of variable
$y=x/\sqrt{1-\lambda}$ and applying \eqref{e:PL_lin}, we get
\[
(1-\lambda)^{n/2}\w(K)\geq 1-\lambda+\lambda\vol(L)
=1-\lambda+\frac{(1-\lambda)^{n/2}}{\lambda^{n-1}}\vol(K).
\]
This finishes the proof.
\end{proof}

By taking $\lambda=1/2$ in \eqref{e:lowerboundWills_classical}, we obtain
the lower bound in Theorem \ref{t:8^n/2vol}:
\begin{equation*}
\w(K)\geq\left(8^{n/2}\vol(K)\right)^{1/2}.
\end{equation*}
Next we show the upper bound in Theorem \ref{t:8^n/2vol}: up to a
factor depending on the standard Gaussian measure $\gamma_n$ of
$K$, the (classical) Wills functional $\w(\cdot)$ can be bounded
from above by $8^{n/2}\vol(K)$. This is the content of the
following result in the more general setting of the generalized
Wills functional $\w_{u}(\cdot\,;E)$.

\begin{theorem}\label{t:W<vol_gamma_n}
Let $K, E\in\K^n$ be convex bodies with non-empty interior such that
$0\in\inter E$ and let $u\in\co(\R_{\geq 0})$ be strictly increasing. Then
\begin{equation}\label{e:GaussianWills_u_E}
\w_{u}(K;E)\leq\min\left\{\binom{2n}{n}e^{-u(0)}\m_n^u,4^n
\m_n^{2u}\right\} \frac{\vol(K)\vol(E)}{\sup_{y\in\R^n}\mu_{u,E}(y-K)},
\end{equation}
where $\mu_{u,E}$ is the measure on $\R^n$ given by
$\dlat\mu_{u,E}(x)=e^{-u(|x|_E)}\,\dlat x$. In particular, for $u=\u$ and
$E=B^n_2$, we have
\begin{equation}\label{e:GaussianWills_classical}
\w(K)\leq\frac{\min\left\{\binom{2n}{n},8^{n/2}\right\}\vol(K)}{\sup_{y\in\R^n}\gamma_n\bigl(y+\sqrt{2\pi}K\bigr)},
\end{equation}
where $C_n:=\min\left\{\binom{2n}{n},8^{n/2}\right\}$ is given by $C_2=6$,
$C_3=20$ and $C_n=8^{n/2}$ for all $n\geq 4$.
\end{theorem}

\begin{proof}
Let $f=e^{-u(|\,\cdot\,|_E)}$ and let $g=\chi_{_K}$. Then, by
Lemma \ref{lem:Asplund}, we have that $f\star g=\f{K}{E}{u}$.
Moreover, $\|f\|_{\infty}=e^{-u(0)}$,
\[
(f*g)(y)=\int_{\R^n}e^{-u(|x|_E)}\chi_{_K}(y-x)\,\dlat x
=\int_{y-K}e^{-u(|x|_E)}\,\dlat x
\]
and
\begin{equation}\label{e:int_f=m_nvol(E)}
\begin{split}
\int_{\R^n}f(x)\,\dlat x
 & =\int_{\R^n} e^{-u(|x|_E)}\dlat x=\int_{\R^n}\int_{u(|x|_E)}^{\infty} e^{-s}\,\dlat s\,\dlat x\\
 & =\int_{u(0)}^{\infty}e^{-s}\int_{\R^n}\chi_{_{\{y:\, u(|y|_E)\leq s\}}}(x)\,\dlat x\,\dlat s\\
 & =\int_{u(0)}^{\infty}e^{-s}\vol\bigl(u^{-1}(s)E\bigr)\,\dlat
 s=\m_n^u\,\vol(E).
\end{split}
\end{equation}
On one hand, by applying \eqref{eq:RSfunctions} to the functions $f$ and
$g$ we get
\begin{equation}\label{e:proving_GaussianWills_u_E}
\sup_{y\in\R^n}\mu_{u,E}(y-K)\,\w_{u}(K;E)\leq\binom{2n}{n}e^{-u(0)}\m_n^u\,\vol(K)\vol(E).
\end{equation}
On the other hand, by applying \eqref{e:diff_function} for $\lambda=1/2$
to the functions $f$ and $g$, we obtain
\[
\begin{split}
\mu_{u/2,E}(-K)\,\w_{u/2}(K;E) & =\int_{-K}e^{-\frac{u}{2}(|x|_E)}\,\dlat x \int_{\R^n}\f{K}{E}{u/2}(x)\,\dlat x\\
 & =\int_{\R^n}\sqrt{f(x)g(-x)}\,\dlat x \int_{\R^n}\sqrt{f\star g(x)}\,\dlat x\\[1mm]
 & \leq 4^n \m_n^u\,\vol(K)\vol(E).
\end{split}
\]
Since both $\w_{u/2}(\cdot\,;E)$ and $\vol(\cdot)$ are translation
invariant we get (from the above inequality for $2u$) that
\begin{equation*}
\sup_{y\in\R^n}\mu_{u,E}(y-K)\,\w_{u}(K;E)\leq 4^n
\m_n^{2u}\,\vol(K)\vol(E),
\end{equation*}
which, together with \eqref{e:proving_GaussianWills_u_E}, shows
\eqref{e:GaussianWills_u_E}.

When $E=B_2^n$ and $u=\u$, since
\[
\m_n^{2\u}=\frac{1}{(2\pi)^{n/2}}\int_{0}^{\infty}e^{-s}s^{n/2}\,\dlat s
=\frac{\m_n^{\u}}{2^{n/2}}=\frac{1}{2^{n/2}}\frac{1}{\kappa_n}
\]
(see \eqref{e:m_i_u_0}) and
$\mu_{\u,B_2^n}(y-K)=\gamma_n\bigl(-\sqrt{2\pi}y+\sqrt{2\pi}K\bigr)$,
then \eqref{e:GaussianWills_u_E} yields
\eqref{e:GaussianWills_classical}. Finally, since
$\binom{2n}{n}\geq\binom{2n}{i}$ for all $i=0,\dots,2n$, we have
that
\[
\binom{2n}{n}\geq\frac{1}{2n+1}\sum_{i=0}^{2n}\binom{2n}{i}=\frac{4^n}{2n+1},
\]
which, jointly with the fact that $2^{n/2}\geq 2n+1$ for all $n\geq 9$,
implies that
\[
C_n=\min\left\{\binom{2n}{n},8^{n/2}\right\}=8^{n/2}\quad\text{ for all }
\; n\geq 9.
\]
Straightforward computations show the remaining cases $n=2,\dots,8$ of the
last assertion. This concludes the proof.
\end{proof}

In view of Theorem \ref{t:W<vol_gamma_n}, the question arises
whether an upper bound for $\w_u(K;E)$ can be obtained involving
$\vol(K)$ as the only functional of $K$. A different technique
will allow us to get additional upper and lower bounds for the
Wills functional just in terms of the volumes of the involved
sets.

On the one hand, it is well-known that the (relative) quermassintegrals of
two convex bodies satisfy the inequalities
\[
\begin{split}
\W_i(K;E)^2\geq\W_{i-1}(K;E)\W_{i+1}(K;E), & \quad\; 1\leq i\leq n-1,\\
\W_i(K;E)\W_j(K;E)\geq\W_k(K;E)\W_l(K;E), &
    \quad\; 0\leq k<i<j<l\leq n,
\end{split}
\]
which are particular cases of the Aleksandrov-Fenchel inequality
(see e.g.~\cite[Section~7.3]{Sch}). Then, we get in particular
that
\begin{equation}\label{e:Wi^2>W0Wn}
\W_i(K;E)^2\geq\W_0(K;E)\W_n(K;E)=\vol(K)\vol(E),
\end{equation}
and hence, for any $u\in\co(\R_{\geq 0})$ strictly increasing,
\[
\begin{split}
\w_u(K;E) & =\sum_{i=0}^n\binom{n}{i}\m_i^u\,\W_i(K;E)\\
 & \geq\vol(K)+\m_n^u\,\vol(E)
    +\bigl(\vol(K)\vol(E)\bigr)^{1/2}\sum_{i=1}^{n-1}\binom{n}{i}\m_i^u.
\end{split}
\]
On the other hand, we see from \eqref{e:Wi^2>W0Wn} that it is not
possible to bound from above the Wills functional only in terms of
the volumes of the involved sets, because there are convex bodies
with volume arbitrarily small but having the remaining
quermassintegrals bounded from below. Therefore, in order that
such an upper bound for $\w_u(K;E)$ makes sense, it is necessary
to `modify' either $K$ or $E$. In this regard we prove the
following result.
\begin{theorem}\label{t:W(K;E)_vol(K)vol(E)}
Let $u\in\co(\R_{\geq 0})$ be strictly increasing and such that $u(0)=0$,
and let $K,E\in\K^n$ with $0\in\inter E$. Then there exist $T\in\SL(n)$
and an absolute constant $C>0$ such that
\[
\w_{u}(TK;E)^{1/n}\leq
C\left(\vol(K)^{1/n}+(\m_n^u)^{1/n}\vol(E)^{1/n}\right).
\]
\end{theorem}

\begin{proof}
We assume without loss of generality that $0\in K$ because the involved
functionals are translation invariant. Since
$\f{K}{E}{u}=e^{-u(|\,\cdot\,|_E)}\star\chi_{_K}$ and both
$e^{-u(|\,\cdot\,|_E)}$ and $\chi_{_K}$ attain their maximums at the
origin and
$\bigl\|e^{-u(|\,\cdot\,|_E)}\bigr\|_{\infty}=\|\chi_{_K}\|_{\infty}=1$,
Theorem~\ref{t:Klartag_Milman} yields the existence of $T_1,T_2\in\SL(n)$
and an absolute constant $C>0$ such that
\begin{align*}
 & \left(\int_{\R^n}\sup_{y\in\R^n}e^{-u\bigl(|T_1y|_E\bigr)}\chi_{_K}(T_2x-T_2y)\,\dlat x\right)^{1/n}\\
 & \hspace*{3cm}=\left(\int_{\R^n}\left(\bigl(e^{-u(|\,\cdot\,|_E)}\circ T_1\bigr)
    \star\bigl(\chi_{_K}\circ T_2\bigr)\right)(x)\,\dlat x\right)^{1/n}\\
 & \hspace*{3cm}\leq C\left[\left(\int_{\R^n}e^{-u\bigl(|T_1x|_E\bigr)}\dlat x\right)^{1/n}
    +\left(\int_{\R^n}\chi_{_K}(T_2x)\,\dlat x\right)^{1/n}\right]
\end{align*}
(see also Lemma \ref{l:asplund_fg}). Then, writing $z=T_1y$ and
$T=T_1T_2^{-1}$, and doing the change of variable $x=T_1^{-1}w$ we get,
using again Lemma \ref{lem:Asplund},
\[
\begin{split}
 \int_{\R^n}\sup_{y\in\R^n}e^{-u\bigl(|T_1y|_E\bigr)} & \chi_{_K}(T_2x-T_2y)\,\dlat x\\
 & =\int_{\R^n}\sup_{z\in\R^n}e^{-u(|z|_E)}\chi_{_K}(T_2T_1^{-1}w-T_2T_1^{-1}z)\,\dlat w\\
 & =\int_{\R^n}\sup_{z\in\R^n}e^{-u(|z|_E)}\chi_{_{T_1T_2^{-1}K}}(w-z)\,\dlat w\\
 & =\int_{\R^n}\left(e^{-u(|\,\cdot\,|_E)}\star\chi_{_{TK}}\right)(w)\,\dlat w
    =\int_{\R^n}\f{TK}{E}{u}(w)\,\dlat w,
\end{split}
\]
and clearly (see \eqref{e:int_f=m_nvol(E)})
\[
\int_{\R^n}e^{-u\bigl(|T_1x|_E\bigr)}\dlat x
    =\int_{\R^n}e^{-u(|x|_E)}\dlat x=\m_n^u\,\vol(E)
\]
and
\[
\int_{\R^n}\chi_{_K}(T_2x)\,\dlat x=\int_{\R^n}\chi_{_K}(x)\,\dlat
x=\vol(K).
\]
Altogether we obtain the result.
\end{proof}

In the particular case of the classical Wills functional we obtain the
following bounds.

\begin{corollary}\label{co:W(K;E)_vol(K)vol(E)}
Let $K\in\K^n$ with $0\in K$. Then there exist $T\in\SL(n)$ and an
absolute constant $C>0$ such that
\[
1+\vol(K)+\bigl(\vol(K)\kappa_n\bigr)^{1/2}\sum_{i=1}^{n-1}\frac{\binom{n}{i}}{\kappa_i}
\leq\w(K)\leq C\left(\vol(T^{-1}K)^{1/n}+1\right)^n.
\]
\end{corollary}

\subsection{Relating $\w(\cdot)$ to other functionals}

Next we reprove McMullen's result, Theorem~\ref{t:properties_Wills} iv),
using a different approach (see also \cite{Vi99} for another proof
established in the setting of Gaussian processes), obtaining also a lower
bound for the Wills functional in terms of the first intrinsic volume and
the circumradius of the set: we show Theorem~\ref{t:McMullen}. To this
end, we will use the so-called Legendre transform of a convex function, as
well as some of its properties. We recall its definition.

The {\it Legendre transform} of a convex function
$f:\R^n\longrightarrow\R\cup\{\infty\}$ is defined by
\[
\L(f)(x)=\sup_{y\in\R^n}\Bigl(\langle x,y\rangle-f(y)\Bigr),
\]
which is also a convex function (see e.g.
\cite[Subsection~1.6.2]{Sch}). Directly related to the Legendre
transform we find the {\it polar function} of a log-concave
function $f:\R^n\longrightarrow\R_{\geq 0}$, which is defined as
\[
f^{\circ}(x)=e^{-\L(-\log f)(x)}.
\]
%We notice that given $K,E\in\K^n$ with $0\in\inter E$,
%\[
%\begin{split}
%\L\bigl(d_E(\cdot,K)^2\bigr)(x) &
%    =\sup_{y\in\R^n}\left(\langle x,y\rangle-\inf_{z\in K}|y-z|_E^2\right)\\
% & =\sup_{y\in\R^n}\sup_{z\in K}\Bigl(\langle x,y\rangle-|y-z|_E^2\Bigr)
%    =\sup_{z\in K}\L\bigl(|\cdot-z|_E^2\bigr)(x).
%\end{split}
%\]
Next we prove that the Legendre transform of the Euclidean
distance from a convex body is closely related to the support
function of the set.

\begin{lemma}\label{l:L(f)}
Let $K\in\K^n$. Then
\[
\L\bigl(\pi d(\,\cdot\,,K)^2\bigr)(x)=\frac{|x|^2}{4\pi}+ h_K(x).
\]
\end{lemma}

\begin{proof}
First we observe that, by the definition of the Legendre
transform,
\begin{equation}\label{e:L(||)}
\begin{split}
\L\bigl(\pi|\,\cdot|^2\bigr)(x) &
    =\sup_{y\in\R^n}\Bigl(\langle y,x\rangle-\pi|y|^2\Bigr)
    =\sup_{r\geq 0}\sup_{|y|=r}\Bigl(\langle y,x\rangle-\pi|y|^2\Bigr)\\
 &  =\sup_{r\geq 0}\bigl(r|x|-\pi r^2\bigr)=\frac{|x|^2}{4\pi}.
\end{split}
\end{equation}
On the other hand,
\[
\begin{split}
\L\bigl(\pi d(\,\cdot\,,K)^2\bigr)(x) &
    =\sup_{y\in\R^n}\!\Bigl(\!\langle y,x\rangle\!-\!\pi\!\inf_{z\in K}\!|y-z|^2\Bigr)\!
    =\!\sup_{y\in\R^n}\sup_{z\in K}\Bigl(\!\langle y,x\rangle\!-\!\pi|y-z|^2\Bigr)\\
 & =\sup_{z\in K}\sup_{y\in\R^n}\Bigl(\langle y,x\rangle-\pi|y|^2-\pi|z|^2+2\pi\langle y,z\rangle\Bigr)\\
 & =\sup_{z\in K}\sup_{y\in\R^n}\Bigl(\langle y,x+2\pi z\rangle-\pi|y|^2-\pi|z|^2\Bigr)\\
 & =\sup_{z\in K}\Bigl(\L\bigl(\pi|\,\cdot|^2\bigr)(x+2\pi z)-\pi|z|^2\Bigr).
\end{split}
\]
Then, using \eqref{e:L(||)} we conclude the result:
\[
\begin{split}
\L\bigl(\pi d(\,\cdot\,,K)^2\bigr)(x) &
    =\sup_{z\in K}\left(\frac{|x+2\pi z|^2}{4\pi}-\pi|z|^2\right)
    =\sup_{z\in K}\left(\frac{|x|^2}{4\pi}+\langle x,z\rangle\right)\\
 & =\frac{|x|^2}{4\pi}+ h_K(x).\qedhere
\end{split}
\]
\end{proof}

We are now in a position to prove Theorem \ref{t:McMullen}.

\begin{proof}[Proof of Theorem \ref{t:McMullen}]
First we note that since all the involved functionals, namely, $\V_1(K)$,
$\cir(K)$ and $\w(K)$, are invariant under translations, we may assume
that the origin is the circumcenter of $K$, i.e., that
$K\subset\cir(K)B_2^n$. Then $|y|\leq\cir(K)$ for all $y\in K$.

We start showing the lower bound. For any convex body $K\in\K^n$
we have
\[
\begin{split}
\w(K) & =\int_{\R^n}e^{-\pi d(x,K)^2}\dlat x
    =\int_{\R^n}e^{-\pi\inf_{y\in K}|x-y|^2}\dlat x\\
 & =\int_{\R^n}e^{-\pi|x|^2}e^{\pi\sup_{y\in K}\bigl(2\langle x,y\rangle-|y|^2\bigr)}\dlat x,
\end{split}
\]
and doing the change of variable $\sqrt{2\pi}\,x=z$ we get
\[
\begin{split}
\w(K) & =\int_{\R^n} \frac{e^{-\frac{|z|^2}{2}}}{(2\pi)^{n/2}}\,
    e^{\sup_{y\in K}\bigl(\sqrt{2\pi}\langle z,y\rangle-\pi|y|^2\bigr)}\dlat z\\
 & =\int_{\R^n}e^{\sup_{y\in K}\bigl(\sqrt{2\pi}\langle
 z,y\rangle-\pi|y|^2\bigr)}\dlat\gamma_n(z).
\end{split}
\]
Then
\[
\begin{split}
\w(K) & \geq\int_{\R^n}e^{\sup_{y\in K}\bigl(\sqrt{2\pi}\langle z,y\rangle
    -\pi\cir(K)^2\bigr)}\dlat\gamma_n(z)\\
 & =e^{-\pi\cir(K)^2}\!\!\int_{\R^n}e^{\sqrt{2\pi}\sup_{y\in K}\langle z,y\rangle}\dlat\gamma_n(z)
    =e^{-\pi\cir(K)^2}\!\!\int_{\R^n}\!e^{\sqrt{2\pi}\,h_K(z)}\dlat\gamma_n(z),
\end{split}
\]
and Jensen's inequality (see e.g. \cite[page~20]{Sch}) applied to
the convex function $e^x$ yields
\[
\w(K)\geq e^{-\pi\cir(K)^2}
    e^{\sqrt{2\pi}\int_{\R^n}h_K(z)\,\dlat\gamma_n(z)}.
\]
So it remains to compute the above integral in the exponential function:
\[
\begin{split}
\int_{\R^n}h_K(z)\,\dlat\gamma_n(z)
 & =\frac{1}{(2\pi)^{n/2}}\int_{\R^n}e^{-\frac{|z|^2}{2}}h_K(z)\,\dlat z\\
 & =\frac{1}{(2\pi)^{n/2}}\int_{\s^{n-1}}\int_0^\infty
    r^ne^{-\frac{r^2}{2}}h_K(u)\,\dlat r\,\dlat u\\
 & =\frac{1}{(2\pi)^{n/2}}\int_{\s^{n-1}}h_K(u)
    \left(\int_0^\infty r^ne^{-\frac{r^2}{2}}\dlat r\right)\dlat u,
\end{split}
\]
and with the change of variable $r^2/2=t$ and using \eqref{e:mean_width}
we get
\begin{equation}\label{e:int_h_k}
\int_{\R^n}h_K(z)\,\dlat\gamma_n(z)
    =\frac{2^{(n-1)/2}}{(2\pi)^{n/2}}\,\Gamma\left(\frac{n-1}{2}+1\right)\kappa_{n-1}\V_1(K)
    =\frac{1}{\sqrt{2\pi}}\V_1(K).
\end{equation}
Therefore, we conclude that
\[
\w(K)\geq e^{-\pi\cir(K)^2}
    e^{\sqrt{2\pi}\frac{1}{\sqrt{2\pi}}\V_1(K)}=e^{\V_1(K)-\pi\cir(K)^2}.
\]

To prove the upper bound we note again that, since $\V_1(K)$ and $\w(K)$
are invariant under translations, we may assume that $0\in K$. We also
observe that, by Lemma \ref{l:L(f)},
\[
\left(\f{K}{B_2^n}{\u}\right)^{\circ}(x)
    =\left(e^{-\pi d(\,\cdot\,,K)^2}\right)^{\circ}(x)
    =e^{-\L\bigl(\pi d(\,\cdot\,,K)^2\bigr)(x)}
    =e^{-\frac{|x|^2}{4\pi}-h_K(x)},
\]
and so, doing the change of variable $x=\sqrt{2\pi}z$ and using Jensen's
inequality and \eqref{e:int_h_k} we get
\[
\begin{split}
\int_{\R^n}\left(\f{K}{B_2^n}{\u}\right)^{\circ}(x)\,\dlat x &
    =\int_{\R^n}e^{-\frac{|x|^2}{4\pi}}e^{-h_K(x)}\dlat x
    =(2\pi)^{n/2}\!\int_{\R^n}\!e^{-\frac{|z|^2}{2}}e^{-\sqrt{2\pi}h_K(z)}\dlat z\\
 & =(2\pi)^n\int_{\R^n}e^{-\sqrt{2\pi}\,h_K(z)}\dlat\gamma_n(z)\\
 & \geq(2\pi)^ne^{-\sqrt{2\pi}\int_{\R^n}h_K(z)\dlat\gamma_n(z)}
    =(2\pi)^ne^{-\V_1(K)}.
\end{split}
\]
Now, since $0\in K$, then
$\bigl\|\f{K}{B_2^n}{\u}\bigr\|_{\infty}=\f{K}{B_2^n}{\u}(0)=1$,
and hence we can apply the functional version of
Blaschke-Santal\'o's inequality (\cite[Theorems~1.1 and 1.2, and
Remark (2) in page~181)]{KM}; see also \cite[Theorem~1.3]{AKM}),
namely,
\[
(2\pi)^n\geq\int_{\R^n}\f{K}{B_2^n}{\u}(x)\,\dlat x
\int_{\R^n}\left(\f{K}{B_2^n}{\u}\right)^{\circ}(x)\,\dlat x,
\]
to obtain
\[
(2\pi)^n\geq\w(K)\int_{\R^n}\left(\f{K}{B_2^n}{\u}\right)^{\circ}(x)\,\dlat
x\geq\w(K)(2\pi)^ne^{-\V_1(K)},
\]
which finishes the proof.
\end{proof}

\subsection{On $0$-symmetric convex bodies in John position}

It is well-known (\cite[Theorem~3]{B}) that the cube $[-1,1]^n$ maximizes
the volume functional among all $0$-symmetric convex bodies in John
position (see also \cite[Theorem~10.13.2]{Sch} and the references
therein). The same property also holds for the mean width (see e.g.
\cite[page~602]{Sch}): $\b(K)\leq\b\bigl([-1,1]^n\bigr)$ for every
$0$-symmetric $K\in\K^n$ which is in John position. Nevertheless, to the
best of the authors' knowledge, similar results for the remaining
intrinsic volumes are not known. Here, following Ball's idea on the
application of the Brascamp-Lieb inequality to data provided by the
conditions in John's theorem, we show that the cube $[-1,1]^n$ also
maximizes the Wills functional among all $0$-symmetric convex bodies in
John position, i.e., we prove Theorem~\ref{t:John}.

\begin{proof}[Proof of Theorem~\ref{t:John}]
Since $K$ is in John position, there exist $c_1,\dots,c_m>0$ and vectors
$v_1,\dots,v_m\in\s^{n-1}\cap\bd K$, $n\leq m\leq n(n+1)/2$, such that the
identity matrix $\I_n=\sum_{i=1}^m c_i v_i\otimes v_i$ (see
\cite[Theorem~10.12.1]{Sch}); here $v_j\otimes v_j$ is the projection in
the direction of $v_j$, i.e., $(v_j\otimes v_j)(x)=\esc{x,v_j}v_j$. In
particular, $\sum_{i=1}^m c_i=n$. From $B_2^n\subset K$ and
$v_i\in\s^{n-1}\cap\bd K$ for $i=1,\dots,m$, we get
$K\subset\bigl\{x\in\R^n:\esc{x,v_i}\leq1\bigr\}$ and thus, since $K$ is
$0$-symmetric, we have $K\subset L=\bigl\{x\in\R^n: |\esc{x,v_i}|\leq
1\bigr\}$. This implies that $\w(K)\leq\w(L)$.

As usual in the literature, we write $\e_i$ to represent the $i$-th
canonical unit vector. Let $\spann\{v\}$ denote the $1$-dimensional linear
subspace spanned by the vector $v\in\R^n$, and let $f_i(t)=e^{-\pi
d(tv_i,P_{\spann\{v_i\}}L)^2}$, $t\in\R$, for $i=1,\dots,m$. Using that
$\sum_{i=1}^mc_i=n$ together with Theorem~\ref{t:properties_Wills} i) and
the fact that $\w\bigl([-v,v]\bigr)=\w\bigl([-\e_i,\e_i]\bigr)$ for all
$v\in\s^{n-1}$ and all $i=1,\dots,n$, one gets
\[
\w\bigl([-1,1]^n\bigr)=\prod_{i=1}^n\w\bigl([-\e_i,\e_i]\bigr)
=\prod_{i=1}^m\w\bigl(P_{\spann\{v_i\}}L\bigr)^{c_i}=\prod_{i=1}^m\left(\int_{\R}f_i(t)\,\dlat
t\right)^{c_i},
\]
and the geometric Brascamp-Lieb inequality (see
\cite[Theorem~10.13.1]{Sch} and the references therein) gives
\begin{equation*}
\begin{split}
\w\bigl([-1,1]^n\bigr)=\prod_{i=1}^m
    \left(\int_{\R}f_i(t)\,\dlat t\right)^{c_i}
 & \geq\int_{\R^n}\prod_{i=1}^m f_i\bigl(\esc{x,v_i}\bigr)^{c_i}\,\dlat x\\
 & =\int_{\R^n}e^{-\pi\sum_{i=1}^m c_i d\bigl(\esc{x,v_i}v_i,P_{\spann\{v_i\}}L\bigr)^2}\,\dlat x.
\end{split}
\end{equation*}
Thus, to conclude the proof, it is enough to show that
\[
\sum_{i=1}^m c_i d\bigl(\esc{x,v_i}v_i, P_{\spann\{v_i\}}L\bigr)^2\leq
d(x,L)^2
\]
because, in that case, we have that
\[
\int_{\R^n}e^{-\pi\sum_{i=1}^m c_i d\bigl(\esc{x,v_i}v_i,
P_{\spann\{v_i\}}L\bigr)^2}\,\dlat x\geq \int_{\R^n}e^{-\pi
d(x,L)^2\,\dlat x}=\w(L)\geq\w(K).
\]
To this end we notice that, for any given $x_0\in L$, all $x\in\R^n$ and
any $i=1,\dots,m$, we have
\[
d\bigl(\esc{x,v_i}v_i,P_{\spann\{v_i\}}L\bigr)^2\leq
d\bigl(\esc{x,v_i}v_i, \esc{x_0,v_i}v_i\bigr)^2 =\esc{x-x_0,v_i}^2
\]
and thus, using that $\I_n=\sum_{i=1}^m c_i v_i\otimes v_i$,
\[
\sum_{i=1}^m c_i d\bigl(\esc{x,v_i}v_i, P_{\spann\{v_i\}}L\bigr)^2
\leq\sum_{i=1}^m c_i\esc{x-x_0,v_i}^2=|x-x_0|^2.
\]
In particular $\sum_{i=1}^m c_i d\bigl(\esc{x,v_i}v_i,
P_{\spann\{v_i\}}L\bigr)^2\leq d(x,L)^2$, as desired.
\end{proof}

\section{Brunn-Minkowski type inequalities for the Wills functional}\label{s:Wills_B-M}

Relating the volume with the Minkowski addition of convex bodies, one is
led to the famous Brunn-Minkowski inequality. One form of it states that
if $K,L\in\K^n$ are convex bodies, and $\lambda\in(0,1)$, then
\begin{equation}\label{e:B-M_ineq}
\vol\bigl((1-\lambda)K+\lambda
L\bigr)^{1/n}\geq(1-\lambda)\vol(K)^{1/n}+\lambda\vol(L)^{1/n},
\end{equation}
with equality, if $\vol(K)\vol(L)>0$, if and only if $K$ and $L$ are
homothetic. In other words, the volume functional $\vol(\cdot)$ is
$(1/n)$-concave.

The above inequality admits a generalization in the context of intrinsic
volumes. Indeed, $\V_i(\cdot)$ is a $(1/i)$-concave functional for all
$i=1,\dots,n$ (see e.g. \cite[Theorem~7.4.5]{Sch}), namely
\[
\V_i\bigl((1-\lambda)K+\lambda L\bigr)^{1/i}
\geq(1-\lambda)\V_i(K)^{1/i}+\lambda\V_i(L)^{1/i}
\]
for any $K,L\in\K^n$ and all $\lambda\in(0,1)$. In particular,
$\V_i(\cdot)$ is $(1/n)$-concave for all $i=0,\dots,n$ (we recall that
$\V_0(K)=1$ for any $K\in\K^n$) and thus one might expect that the same
holds for the Wills functional $\w(\cdot)=\sum_{i=0}^n\V_i(\cdot)$. In
general, this is not the case, as the following example shows.

\begin{example}\label{ex:Wills_not_1/n_BM}
The $i$-th intrinsic volume of the Euclidean ball is
$\V_i(rB_2^n)=\binom{n}{i}/\kappa_{n-i}\W_{n-i}(rB_2^n)=\binom{n}{i}\kappa_n/\kappa_{n-i}\,r^i$.
Then, one can see that the inequality
\[
\w\left(\frac{r+R}{2}B_2^n\right)^{1/n}\geq
\frac{1}{2}\w\left(rB_2^n\right)^{1/n}+\frac{1}{2}\w\left(RB_2^n\right)^{1/n}
\]
is, in general, not true, just taking $r=1$, $R=2$ and $n=2,3,4\dots$

The Euclidean balls $B_2^n$ and $2B_2^n$ also show that the additive
version of the Brunn-Minkowski inequality for the Wills functional,
namely,
\[
\w\left(B_2^n+2B_2^n\right)^{1/n}\geq
\w\left(B_2^n\right)^{1/n}+\w\left(2B_2^n\right)^{1/n}
\]
is, in general, not true.
\end{example}

Although, as seen, the Wills functional is not a $(1/n)$-concave
functional, it is not ``far'' from being so. Indeed, when dealing
with $K$ and $L$ orthogonal boxes we have
$\w\bigl((1-\lambda)K+\lambda
L\bigr)^{1/n}\geq(1-\lambda)\w(K)^{1/n}+\lambda\w(L)^{1/n}$: this
can be shown as a direct consequence of
Theorem~\ref{t:properties_Wills} i), using the
arithmetic-geometric mean inequality, and the fact that the Wills
functional of a segment $\ell$ is $\w(\ell)=1+\vol_1(\ell)$.
Moreover, by adding the additional constant $1/(n!)^{1/n}$ on the
right-hand side of the above inequality, it becomes true for
arbitrary convex bodies, as Theorem \ref{t:BM_W_constant} shows.
We notice that this constant is of the order of $e/n$. We state
and prove it in the more general setting of the generalized Wills
functional.
\begin{theorem}
Let $K, L, E\in\K^n$ with $0\in\inter E$, $\lambda\in(0,1)$ and
$u\in\co(\R_{\geq 0})$ strictly increasing. Then
\begin{equation}\label{e_BM_Wills_Berw}
\w_{u}\bigl((1-\lambda)K+\lambda L;E\bigr)^{1/n}
\geq\!\frac{e^{-\frac{n-1}{n}u(0)}}{(n!)^{1/n}}\left(\!(1\!-\!\lambda)\w_{u}(K;E)^{1/n}\!
+\!\lambda\w_{u}(L;E)^{1/n}\!\right).
\end{equation}
\end{theorem}

\begin{proof}
For the sake of brevity we set $M_\lambda=(1-\lambda)K+\lambda L$. The
Brunn-Minkowski inequality \eqref{e:B-M_ineq} implies that
\[
\vol\bigl(M_\lambda+u^{-1}(s)E\bigr)^{1/n}
\geq(1-\lambda)\vol\bigl(K+u^{-1}(s)E\bigr)^{1/n}+\lambda\vol\bigl(L+u^{-1}(s)E\bigr)^{1/n}
\]
for all $s\geq u(0)$, and then we clearly get
\begin{equation}\label{e_provingBM_Berw_1}
\begin{split}
 \int_{u(0)}^\infty\!\vol\bigl(M_\lambda+u^{-1}(s)E\bigr)^{1/n}e^{-s}\,\dlat s
 & \geq(1-\lambda)\!\int_{u(0)}^\infty\!\vol\bigl(K+u^{-1}(s)E\bigr)^{1/n}e^{-s}\,\dlat s\\
 & \quad +\lambda\int_{u(0)}^\infty\vol\bigl(L+u^{-1}(s)E\bigr)^{1/n}e^{-s}\,\dlat s.
\end{split}
\end{equation}
On one hand, by Jensen's inequality, together with
\eqref{e:ComputationIntegral}, we have
\begin{equation}\label{e_provingBM_Berw_2}
\begin{split}
 \int_{u(0)}^\infty\vol\bigl(M_\lambda+u^{-1}(s)E\bigr)^{1/n}e^{-s}\,\dlat s
 & \leq\left(\int_{u(0)}^\infty\vol\bigl(M_\lambda+u^{-1}(s)E\bigr)e^{-s}\,\dlat s\right)^{1/n}\\
 & =\w_{u}(M_\lambda;E)^{1/n};
\end{split}
\end{equation}
on the other hand, doing the change of variable $t=s-u(0)$, we can apply
the extension of Berwald's inequality (cf. \cite[Satz 8]{Ber}) which is
proved in \cite[Lemma~2.1]{ABG2}, and with \eqref{e:ComputationIntegral}
we obtain
\begin{equation}\label{e_provingBM_Berw_3}
\begin{split}
\int_{u(0)}^\infty\!\!\vol\bigl(K\!+\!u^{-1}(s)E\bigr)^{1/n}e^{-s}\dlat s
 & \geq\!\frac{e^{-\frac{n-1}{n}u(0)}}{(n!)^{1/n}}\!\left(\!\int_{u(0)}^\infty\!\!\!\vol\bigl(K\!+\!u^{-1}(s)E\bigr)e^{-s}\dlat s\!\right)^{\!\!1/n}\\
 & =\frac{e^{-\frac{n-1}{n}u(0)}}{(n!)^{1/n}}\,\w_{u}(K;E)^{1/n}.
\end{split}
\end{equation}
The same holds for $L$ in place of $K$. Altogether, from
\eqref{e_provingBM_Berw_1}, \eqref{e_provingBM_Berw_2} and
\eqref{e_provingBM_Berw_3}, we get \eqref{e_BM_Wills_Berw}.
\end{proof}

In the case of the classical Wills functional, from
\eqref{e_BM_Wills_Berw} for $u=\u$ and $E=B_2^n$, we get
Theorem~\ref{t:BM_W_constant}:
\[
\w\bigl((1-\lambda)K+\lambda L\bigr)^{1/n}
\geq\frac{1}{(n!)^{1/n}}\left((1-\lambda)\w(K)^{1/n}
+\lambda\w(L)^{1/n}\right).
\]

Aiming to get a ``real'' concavity property for the Wills functional, we
will exploit its integral formula via a log-concave function (cf.
\eqref{e:ComputationIntegral}) as well as the corresponding machinery: the
Pr\'ekopa-Leindler inequality (Theorem~\ref{t:BBL} for $p=0$). This gives
rise to Theorem~\ref{t:BM_mult_W}. Again we state and prove it in the
general setting: we see that the generalized Wills functional is
log-concave.

\begin{theorem}
Let $E\in\K^n$ with $0\in\inter E$ and let $u\in\co(\R_{\geq 0})$. Then
$\w_{u}(\,\cdot\,;E)$ is log-concave, i.e., for any $K,L\in\K^n$ and all
$\lambda\in(0,1)$,
\begin{equation}\label{e:BMWills_u_E}
\w_{u}\bigl((1-\lambda)K+\lambda
L;E\bigr)\geq\w_{u}(K;E)^{1-\lambda}\w_{u}(L;E)^\lambda.
\end{equation}
\end{theorem}

\begin{proof}
By the triangle inequality for $|\,\cdot|_E$ we have that
\begin{equation}\label{e:relation_distances}
d_E\bigl((1-\lambda)x+\lambda y, (1-\lambda)K+\lambda L\bigr)
\leq(1-\lambda)d_E(x,K)+\lambda d_E(y,L)
\end{equation}
for every $x,y\in\R^n$, and hence, from the convexity and monotonicity of
$u$,
\[
u\Bigl(d_E\bigl((1-\lambda)x+\lambda y, (1-\lambda)K+\lambda L\bigr)\Bigr)
\leq (1-\lambda)u\bigl(d_E(x,K)\bigr)+\lambda u\bigl(d_E(y,L)\bigr).
\]
Therefore, the functions $f=\f{K}{E}{u}$, $g=\f{L}{E}{u}$ and
$h=\f{(1-\lambda)K+\lambda L}{E}{u}$ are in the conditions of the
Pr\'ekopa-Leindler inequality (Theorem~\ref{t:BBL} for $p=0$), and thus
\eqref{e:PrekopaLeindler} yields \eqref{e:BMWills_u_E}.
\end{proof}

\begin{remark}\label{r:equal_proj}
Under the assumption that the involved convex bodies have a common
projection onto a hyperplane, the concavity of the Wills
functional can be improved. More precisely, if $K,L\in\K^n$ are
convex bodies with non-empty interior such that $P_HK=P_HL$ for
some hyperplane $H\in\G(n,n-1)$ then, for any $E\in\K^n$ with
$0\in\inter E$, all $\lambda\in(0,1)$ and any $u\in\co(\R_{\geq
0})$,
\[
\w_{u}\bigl((1-\lambda)K+\lambda L;E\bigr)\geq(1-\lambda)\w_{u}(K;E)+\lambda\w_{u}(L;E).
\]
Indeed, this follows from the concavity of the relative
quermassintegrals in this setting (see e.g.
\cite[Theorem~7.7.2]{Sch}) jointly with
Proposition~\ref{lem:ComputationIntegral}, namely,
$\w_{u}(\,\cdot\,;E)=\sum_{i=0}^n\binom{n}{i}\m_i^u\W_i(\,\cdot\,;E)$.
\end{remark}

Coming back to the classical Wills functional, from \eqref{e:BMWills_u_E}
for $u=\u$ and $E=B_2^n$, we have Theorem~\ref{t:BM_mult_W}:
\[
\w\bigl((1-\lambda)K+\lambda L\bigr)\geq\w(K)^{1-\lambda}\w(L)^\lambda.
\]
Without any extra assumption on the convex bodies (cf.
Remark~\ref{r:equal_proj}), this (log-)concavity seems not
possible to be improved: considering again the Wills functionals
of Euclidean balls (see Example~\ref{ex:Wills_not_1/n_BM}),
numerical computations show that for the balls $0.2B_2^n$ and
$0.05B_2^n$ (with $\lambda=1/2$), the $(1/(n+1))$-concavity does
not hold in general for $n>2$, i.e.,
\[
\w\left(\frac{0.25}{2}B_2^n\right)^{1/(n+1)}\!\!
<\frac{1}{2}\w\left(0.2B_2^n\right)^{1/(n+1)}+\frac{1}{2}\w\left(0.05B_2^n\right)^{1/(n+1)}
\;\text{ for } n>2;
\]
moreover, these numerical calculations for the previous balls
suggest that for any $p>3$, there exists a value of the dimension
$n<p$ such that the $(1/p)$-concavity does not hold.

However, dimension $n=2$ is a singular case: here, the classical Wills
functional is $(1/(n+1))$-concave; we recall that, even in the planar
case, the Wills functional is not $(1/n)$-concave, as mentioned in Example
\ref{ex:Wills_not_1/n_BM}. This is the content of the following result, in
which we exploit a suitable generalized Wills functional to derive
additional information for the classical one.

\begin{theorem}\label{t:planarWillsBM}
Let $K,L\in\K^2$ and $\lambda\in(0,1)$. Then
\begin{equation}\label{e:planarWillsBM}
\w\bigl((1-\lambda)K+\lambda
L\bigr)^{1/3}\geq(1-\lambda)\w(K)^{1/3}+\lambda\w(L)^{1/3}.
\end{equation}
\end{theorem}

\begin{proof}
First we assume that there exists an increasing non-negative continuous
function $\phi:[0,a]\longrightarrow\R_{\geq0}$, for some $a>0$, such that
\begin{equation}\label{eqs_moments}
\int_0^a\phi(t)t^i\,\dlat t=\frac{1}{\kappa_i}, \quad \text{ for }
i=0,1,2.
\end{equation}
Then, the function $G:\R_{\geq0}\longrightarrow\R_{\geq0}$ defined by
$G(t)=\mu\bigl([t,\infty)\bigr)$, where $\mu$ is the measure on
$\R_{\geq0}$, concentrated on $[0,a]$, given by
$\dlat\mu(t)=\phi(t)\,\dlat t$, satisfies that (see
\cite[Lemma~1.1]{HCY15} and the references therein)
\begin{equation}\label{e:relationWills_G}
\int_{\R^2}G\bigl(d(x,M)\bigr)\,\dlat x=\w(M)
\end{equation}
for any convex body $M\in\K^2$. We consider the functions
$f=G\bigl(d(\,\cdot\,,K)\bigr)$, $g=G\bigl(d(\,\cdot\,,L)\bigr)$
and $h=G\bigl(d\bigl(\,\cdot\,,(1-\lambda) K+\lambda
L\bigr)\bigr)$. By the fundamental theorem of calculus, the
derivative of $G$ fulfils $G'(t)=-\phi(t)$ for all $t\in[0,a]$,
which, from the monotonicity of $\phi$, implies that $G$ is
concave on $[0,a]$. This, together with
\eqref{e:relation_distances} for $E=B_2^n$ and the fact that $G$
is a decreasing function, yields $h\bigl((1-\lambda)x+\lambda
y\bigr)\geq (1-\lambda)f(x)+\lambda g(y)$ for all $x,y\in\R^2$
such that $f(x)g(y)>0$. Thus, by the Borell-Brascamp-Lieb
inequality (Theorem \ref{t:BBL}), we get that
\begin{equation*}
\begin{split}
\int_{\R^2}G\Bigl( & d\bigl(x,(1-\lambda) K+\lambda L\bigr)\Bigr)\,\dlat x\\
 & \geq\left[(1-\lambda)\left(\int_{\R^2}G\bigl(d(x,K)\bigr)\,\dlat x\right)^{1/3}
    +\lambda\left(\int_{\R^2}G\bigl(d(x,L)\bigr)\,\dlat x\right)^{1/3}\right]^3,
\end{split}
\end{equation*}
which, using \eqref{e:relationWills_G}, gives \eqref{e:planarWillsBM}.

To conclude we prove the existence of such a function $\phi$. We define it
by $\phi(t)=a_1+a_2t+a_3t^2$, for suitable $a_1,a_2,a_3\in\R$ to be
determined later on. Then, conditions \eqref{eqs_moments} yield the system
of linear equations in $(a_1,a_2,a_3)$
\[
\left\{
\begin{array}{ccccc}
1&=&\displaystyle\int_0^a\phi(t)\,\dlat t&=& \displaystyle a_1a+\frac{a_2a^2}{2}+\frac{a_3a^3}{3},\\[3mm]
\displaystyle \frac{1}{2}&=&\displaystyle\int_0^a\phi(t)t\,\dlat t &=& \displaystyle\frac{a_1a^2}{2}+\frac{a_2a^3}{3}+\frac{a_3a^4}{4},\\[3mm]
\displaystyle\frac{1}{\pi}&=&\displaystyle\int_0^a\phi(t)t^2\,\dlat t &=&
\displaystyle\frac{a_1a^3}{3}+\frac{a_2a^4}{4}+\frac{a_3a^5}{5},
\end{array}\right.
\]
whose solution is
\[
\begin{split}
a_1 & =\frac{9\pi a^2-18\pi a+30}{\pi a^3},\quad
    a_2=\frac{-36\pi a^2+96\pi a-180}{\pi a^4},\\
a_3 & =\frac{30\pi a^2-90\pi a+180}{\pi a^5}.
\end{split}
\]
By studying the sign of the above quadratic polynomials in the variable
$a$, one finds we can choose an appropriate value of $a$, for instance
$a=0.91$, so that $a_i>0$ for all $i=1,2,3$; therefore, $\phi$ is
non-negative and increasing on $[0,a]$. This finishes the proof.
\end{proof}

We conclude this section by showing that, although the additive version of
the Brunn-Minkowski inequality does not hold for the Wills functional (cf.
Example~\ref{ex:Wills_not_1/n_BM}), $\w(\cdot)$ satisfies a reverse
Brunn-Minkowski inequality with exponent $1/n$: we prove
Theorem~\ref{t:reverse_BM_W}. Indeed, we state and show it in the more
general setting of the generalized Wills functional.
\begin{theorem}\label{t:reverse_BM_W_gen}
Let $K,L,E\in\K^n$ with $0\in\inter E$ and let $u\in\co(\R_{\geq 0})$ with
$u(0)=0$. Then there exist $T\in\SL(n)$ and an absolute constant $C>0$
such that
\begin{equation*}
\w_u(K+TL;E)^{1/n}\leq C\left[\w_u(K;E)^{1/n}+\w_u(L;E)^{1/n}\right].
\end{equation*}
\end{theorem}

\begin{proof}
We may assume without loss of generality that $0\in K\cap L$ because the
involved functionals are invariant under translations.

Since
$\bigl\|\f{K}{E}{u}\bigr\|_{\infty}=\bigl\|\f{L}{E}{u}\bigr\|_{\infty}=\f{K}{E}{u}(0)=\f{L}{E}{u}(0)=1$,
Theorem~\ref{t:Klartag_Milman} yields the existence of $T_1,T_2\in\SL(n)$
and an absolute constant $C>0$ such that
\[
\begin{split}
\biggl(\int_{\R^n}\bigl((\f{K}{E}{u} & \circ T_1)\star(\f{L}{E}{u}\circ T_2)\bigr)(x)\,\dlat x\biggr)^{1/n}\\
 & \leq C\left[\left(\int_{\R^n}(\f{K}{E}{u}\circ T_1)(x)\,\dlat x\right)^{1/n}
    +\left(\int_{\R^n}(\f{L}{E}{u}\circ T_2)(x)\,\dlat x\right)^{1/n}\right].
\end{split}
\]
On one hand, denoting by $f=\f{K}{E}{u}\circ T_1$ and $g=\f{L}{E}{u}\circ
T_2$, we get
\[
\int_{\R^n}f(x)\,\dlat x=\int_{\R^n}e^{-u\bigl(d_E(T_1x,K)\bigr)}\dlat x
=\int_{\R^n}e^{-u\bigl(d_E(x,K)\bigr)}\dlat x=\w_u(K;E)
\]
and, analogously,
\[
\int_{\R^n}g(x)\,\dlat x=\w_u(L;E).
\]
On the other hand, writing $T=T_2T_1^{-1}$ and using
Lemma~\ref{l:asplund_fg} and Lemma~\ref{lem:Asplund}, we have
\[
\begin{split}
\int_{\R^n}(f\star g)(x)\,\dlat x &
    =\int_{\R^n}\bigl(\f{K}{E}{u}\star(\f{L}{E}{u}\circ T)\bigr)\bigl(T_1(x)\bigr)\,\dlat x\\
 & =\int_{\R^n}\bigl(\f{K}{E}{u}\star(\f{L}{E}{u}\circ T)\bigr)(x)\,\dlat x\\
 & =\int_{\R^n}\biggl(\f{K}{E}{u}\star
    \Bigl(\bigl(e^{-u(|\,\cdot\,|_E)}\circ T\bigr)\star(\chi_{_L}\circ T)\Bigr)\biggr)(x)\,\dlat x\\
 & =\int_{\R^n}\biggl(\bigl(e^{-u(|\,\cdot\,|_E)}\star\chi_{_K}\bigr)\star
    \Bigl(\bigl(e^{-u(|\,\cdot\,|_E)}\circ T\bigr)\star\chi_{_{T^{-1}L}}\Bigr)\!\biggr)(x)\,\dlat x.
\end{split}
\]
Now, since
\[
\left(e^{-u(|\,\cdot\,|_E)}\star
e^{-u(|T(\cdot)|_E)}\right)(x)\geq
e^{-u(|x|_E)}e^{-u(|T(0)|_E)}=e^{-u(|x|_E)},
\]
we have, from Lemma~\ref{lem:Asplund}, that
\[
\begin{split}
\bigl(e^{-u(|\,\cdot\,|_E)}\!\star\chi_{_K}\bigr)\!\star\!
    \Bigl(\!\bigl(e^{-u(|\,\cdot\,|_E)}\!\circ T\bigr)\star\chi_{_{T^{-1}L}}\!\Bigr)
 & =e^{-u(|\,\cdot\,|_E)}\star e^{-u(|T(\cdot)|_E)}\star\chi_{_{K+T^{-1}L}}\\
 & \geq e^{-u(|\,\cdot\,|_E)}\star\chi_{_{K+T^{-1}L}}\!=\f{K+T^{-1}L}{E}{u},
\end{split}
\]
and so we get
\[
\int_{\R^n}(f\star g)(x)\,\dlat x
\geq\int_{\R^n}\f{K+T^{-1}L}{E}{u}(x)\,\dlat
x=\w_u\bigl(K+T^{-1}L;E\bigr).
\]
Altogether concludes the proof.
\end{proof}

\section{Rogers-Shephard type inequalities for the Wills functional}

In this last section we obtain Rogers-Shephard type inequalities for the
Wills functional. First we study section/projection Rogers-Shephard type
relations.

\subsection{The Wills functional for projections and sections}

We start proving Theorem \ref{t:R-S_Wills}, for which we use the following
result that we state in the general setting of the generalized Wills
functional.

\begin{proposition}\label{t:R-S_Wills_gen}
Let $K,E\in\K^n$ with $0\in\inter E$, let $u\in\co(\R_{\geq 0})$ and
$H\in\G(n,k)$. Then
\[
\begin{split}
\wk{k}{u}(P_HK;P_HE)\,\wk{n-k}{u}(K\cap H^{\bot};E\cap H^{\bot})
\leq\binom{n}{k}e^{-u(0)}\w_{u}(K;E).
\end{split}
\]
\end{proposition}

\begin{proof}
Using Lemma~\ref{l:P_Hf_K=f_P_HK} and applying
\eqref{e:R-S_f_log-concave1} to $\f{K}{E}{u}$ we get
\[
\begin{split}
\wk{k}{u}(P_HK;P_HE & )\,\wk{n-k}{u}(K\cap H^{\bot};E\cap H^{\bot})\\
 & =\int_{H}\f{P_HK}{P_HE}{u}(x)\,\dlat x\int_{H^{\bot}}\f{K\cap H^{\bot}}{E\cap H^{\bot}}{u}(y)\,\dlat y\\
 & \leq\int_{H}\bigl(P_H\f{K}{E}{u}\bigr)(x)\,\dlat x\int_{H^{\bot}}\f{K}{E}{u}(y)\,\dlat y\\
 & \leq\binom{n}{k}\|\f{K}{E}{u}\|_{\infty}\!\int_{\R^n}\!\!\f{K}{E}{u}(z)\,\dlat z=\!\binom{n}{k}e^{-u(0)}\w_{u}(K;E).\qedhere
\end{split}
\]
\end{proof}

We are now in a position to prove Theorem \ref{t:R-S_Wills}.

\begin{proof}[Proof of Theorem \ref{t:R-S_Wills}]
On one hand, taking into account that the classical Wills functional does
not depend on the dimension of the embedding space, a first upper bound is
obtained from Proposition~\ref{t:R-S_Wills_gen} applied to
$\f{K}{B_2^n}{\u}$:
\begin{equation}\label{e:bound1_R-S_Wills}
\w(P_HK)\w(K\cap H^{\bot})\leq\binom{n}{k}\w(K).
\end{equation}
On the other hand, since
$\bigl\|\f{K}{B_2^n}{\u}\bigr\|_{\infty}=\f{K}{B_2^n}{\u}(0)$, using
\eqref{e:R-S_f_log-concave2} we get
\[
(1-\lambda)^k\lambda^{n-k}\int_H\left(P_H\f{K}{B_2^n}{\u}\right)(x)^{1-\lambda}
\dlat x\int_{H^{\bot}}\f{K}{B_2^n}{\u}(y)^{\lambda}\dlat y
\leq\int_{\R^n}\f{K}{B_2^n}{\u}(z)\,\dlat z.
\]
Now, since (cf. Lemma \ref{l:P_Hf_K=f_P_HK})
\[
\begin{split}
\left(P_H\f{K}{B_2^n}{\u}\right)(x)^{1-\lambda} &
    =\f{P_HK}{B_2^k}{\u}(x)^{1-\lambda}
    =e^{-(1-\lambda)\pi d(x,P_HK)^2}\\
 & =e^{-\pi d\bigl(\sqrt{1-\lambda}\,x,\sqrt{1-\lambda}\,P_HK\bigr)^2}
    =\f{\sqrt{1-\lambda}\,P_HK}{B_2^k}{{\u}}\left(\sqrt{1-\lambda}\,x\right)
\end{split}
\]
for every $x\in H$ and
\[
\begin{split}
\f{K}{B_2^n}{\u}(y)^{\lambda} & \geq\f{K\cap
H^{\bot}}{B_2^{n-k}}{{\u}}(y)^{\lambda}
    =e^{-\lambda\pi d(y,K\cap H^\perp)^2}\\
 & =e^{-\pi d\bigl(\sqrt{\lambda}\,y,\,\sqrt{\lambda}\,K\cap H^\perp\bigr)^2}
    =\f{\sqrt{\lambda}\,K\cap H^{\bot}}{B_2^{n-k}}{{\u}}\left(\sqrt{\lambda}\,y\right)
\end{split}
\]
for all $y\in H^\perp$, doing the suitable change of variable in each
integral we get
\[
\begin{split}
\frac{(1-\lambda)^k\lambda^{n-k}}{(1-\lambda)^{k/2}\lambda^{(n-k)/2}}
 & \int_H\f{\sqrt{1-\lambda}\,P_HK}{B_2^k}{\u}(x)\,\dlat x
    \int_{H^{\bot}}\f{\sqrt{\lambda}\,K\cap H^{\bot}}{B_2^{n-k}}{\u}(y)\,\dlat y\\
 & \leq\int_{\R^n}\f{K}{B_2^n}{\u}(z)\,\dlat z,
\end{split}
\]
this is,
\begin{equation}\label{e:R-S_f_K_lambda}
(1-\lambda)^{k/2}\lambda^{(n-k)/2}\,\w\Bigl(\sqrt{1-\lambda}\,P_HK\Bigr)
\w\Bigl(\sqrt{\lambda}\,K\cap H^{\bot}\Bigr)\leq \,\w(K).
\end{equation}
Then, taking $\lambda=1/2$,
\[
\frac{1}{2^{n/2}}\w\left(\frac{1}{\sqrt{2}}P_HK\right)
\w\left(\frac{1}{\sqrt{2}}K\cap H^{\bot}\right)\leq\w(K),
\]
or equivalently,
\[
\w\left(P_HK\right)\w\bigl(K\cap H^{\bot}\bigr)\leq
2^{n/2}\w\bigl(\sqrt{2}K\bigr).
\]
Together with \eqref{e:bound1_R-S_Wills} we get the result.
\end{proof}

\begin{remark}
Since the maximum of the function
$(1-\lambda)^{k/2}\lambda^{(n-k)/2}$ when $\lambda\in(0,1)$ is
attained for $\lambda=(n-k)/n$, the best inequality which can be
obtained from \eqref{e:R-S_f_K_lambda} would be
\[
\w\Bigl(\left(\tfrac{k}{n}\right)^{1/2}P_HK\Bigr)\,
  \w\Bigl(\left(\tfrac{n-k}{n}\right)^{1/2} K\cap H^{\bot}\Bigr)
  \leq\frac{n^{n/2}}{k^{k/2}(n-k)^{(n-k)/2}}\w(K).
\]
\end{remark}

\begin{remark}
The minimum in Theorem~\ref{t:R-S_Wills} may be attained in both values,
even for the same sets, depending on $k$. For instance, if we consider the
unit cube $K=[0,1]^n$ and $E=B^n_2$, since $\W_i(K)=\kappa_i$ for
$i=0,\dots,n$, then
\[
\binom{n}{k}\w(K)=\binom{n}{k}2^n\quad \text{ and }\quad
2^{n/2}\,\w\bigl(\sqrt{2}K\bigr)=\sum_{i=0}^n\binom{n}{i}2^{n-i/2}.
\]
In dimension $n=10$, if $k=5$ then
$2^{n/2}\,\w\bigl(\sqrt{2}K\bigr)<\binom{n}{k}\w(K)$, whereas we
get the opposite inequality when $k\neq 5$.
\end{remark}

We conclude this subsection by showing some relations of the Wills
functional of a convex body in terms of the Wills functional of certain
projections of it onto hyperplanes. First, we recall some auxiliary
results: if $f:\R^n\longrightarrow\R_{\geq0}$ is a log-concave and
integrable function, the polar projection body of $f$, $\Pi^*f$, was
introduced in \cite{ABG}, via the norm induced, by
\begin{equation}\label{e:def_Pi^*f}
|v|_{_{\Pi^*f}}=2\int_{H_{v,0}}(P_{H_{v,0}}f)(x)\,\dlat x
\quad\text{ for all }\; v\in\s^{n-1}.
\end{equation}
Then, using polar coordinates, we have
\begin{equation}\label{e:Pi^*f}
\vol\bigl(\Pi^*f\bigr)=\kappa_n\int_{\s^{n-1}}\frac{1}{|v|_{_{\Pi^*f}}^n}\,\dlat\sigma(v),
\end{equation}
where $\sigma$ denotes the Lebesgue probability measure on
$\s^{n-1}$ (cf.~\cite[(1.53)]{Sch}). Regarding a lower bound for
$\vol\bigl(\Pi^*f\bigr)$, it was shown  in \cite[Theorem~1.1]{ABG}
that
\begin{equation}\label{e:funct_Zhang}
\int_{\R^n}\int_{\R^n}\min\bigl\{f(x),f(y)\bigr\}\,\dlat x\,\dlat
y\leq2^n n!\,\|f\|_{\infty}\|f\|_1^{n+1}\,\vol\bigl(\Pi^*f\bigr).
\end{equation}
Moreover, by the so-called affine Sobolev inequality (see
\cite[Theorem~1.1]{Z} and \cite[page~2]{ABG}), one has
\begin{equation}\label{e:AffineSobolev}
2^n\|f\|^n_{_{\frac{n}{n-1}}}\vol\bigl(\Pi^*f\bigr)\kappa_{n-1}^n\leq
\kappa_n^n.
\end{equation}
Using the above relations we can derive the maximal and minimal
values of the Wills functional of the projections onto hyperplanes
of a convex body, in terms of the Wills functional of the original
set.

\begin{theorem}\label{t:Wills_max_min_sigma}
Let $K\in\K^n$. Then
\begin{equation}\label{e:Wills_max_min}
\begin{split}
\max_{v\in\s^{n-1}}\w\left(P_{H_{v,0}} K\right) & \geq 2C_n\w(K)^{(n-1)/n} \quad\text{ and}\\
\min_{v\in\s^{n-1}}\w\left(P_{H_{v,0}} K\right) & \leq
D_n\w(K)^{(n-1)/n},
\end{split}
\end{equation}
where $C_n$ and $D_n$ are given by
\[
C_n=\frac{1}{2}\left(\sqrt{\frac{n-1}{n}}\,\right)^{n-1}\frac{\kappa_{n-1}}{\kappa_n^{(n-1)/n}}
\quad\text{ and }\quad D_n=(n!)^{1/n}\kappa_n^{1/n}.
\]
Moreover,
\begin{equation}\label{e:Wills_sigma}
\sigma\left(v\in\s^{n-1}: \w\left(P_{H_{v,0}} K\right)\geq
C_n\w(K)^{(n-1)/n}\right)\geq 1-\frac{1}{2^n}.
\end{equation}
In particular, there exist absolute constants $C$ and $D$ such
that
\begin{equation*}
\begin{split}
\max_{v\in\s^{n-1}} \w\left(P_{H_{v,0}} K\right) & \geq C\,\w(K)^{(n-1)/n},\\
\min_{v\in\s^{n-1}} \w\left(P_{H_{v,0}} K\right) & \leq D\sqrt{n}\,\w(K)^{(n-1)/n}\quad\text{ and}\\
\sigma\Bigl(v\in\s^{n-1}: \w\left(P_{H_{v,0}} K\right) &
    \geq C\,\w(K)^{(n-1)/n}\Bigr)\geq 1-\frac{1}{2^n}
\end{split}
\end{equation*}
for $n$ large enough.
\end{theorem}

\begin{proof}
For the sake of brevity we write $f=\f{K}{B_2^n}{\u}$. On one hand, from
\eqref{e:def_Pi^*f} and Lemma \ref{l:P_Hf_K=f_P_HK} we obtain
\begin{equation}\label{e:vPi^*f}
|v|_{_{\Pi^*f}}=2\w\left(P_{H_{v,0}} K\right).
\end{equation}
On the other hand, since $\|f\|_\infty=1$, then
\[
f(x)f(y)\leq\min\bigl\{f(x),f(y)\bigr\}\quad\text{ for all }\;x,y\in\R^n.
\]
This fact, jointly with \eqref{e:funct_Zhang}, implies that
\[
1\leq2^n n!\,\|f\|_1^{n-1}\,\vol\bigl(\Pi^*f\bigr)
\]
and hence, together with \eqref{e:AffineSobolev}, we get
\begin{equation*}
\frac{1}{2(n!)^{1/n}\|f\|_1^{(n-1)/n}}\leq\vol\bigl(\Pi^*f\bigr)^{1/n}\leq
\frac{\kappa_n}{2\kappa_{n-1}\|f\|_{_{\frac{n}{n-1}}}}.
\end{equation*}
Taking into account \eqref{e:Pi^*f}, \eqref{e:f_n/n-1} and
\eqref{e:vPi^*f}, the above inequality yields
\begin{equation*}
\begin{split}
\frac{1}{(n!)^{1/n}\kappa_n^{1/n}\w(K)^{(n-1)/n}} &
    \leq\left(\int_{\s^{n-1}}\frac{1}{\w\left(P_{H_{v,0}} K\right)^n}\,\dlat\sigma(v)\right)^{1/n}\\
 & \leq \frac{\kappa_n^{(n-1)/n}}{\kappa_{n-1}\left(\sqrt{\frac{n-1}{n}}\right)^{n-1}\w\left(\sqrt{\frac{n}{n-1}}\,K\right)^{(n-1)/n}},
\end{split}
\end{equation*}
and then we infer that
\[
\max_{v\in\s^{n-1}}\w\left(P_{H_{v,0}} K\right)\geq
2C_n\w\left(\sqrt{\frac{n}{n-1}}\,K\right)^{(n-1)/n}
\]
and
\[
\min_{v\in\s^{n-1}} \w\left(P_{H_{v,0}} K\right)\leq D_n\w(K)^{(n-1)/n},
\]
which imply \eqref{e:Wills_max_min}, from the monotonicity (and the
translation invariance) of the classical Wills functional.

To prove \eqref{e:Wills_sigma}, we observe that \eqref{e:Pi^*f} and
Markov's inequality (see e.g. \cite[Proposition~2.3.10]{Cohn}) imply that,
for all $t>0$,
\begin{equation*}\label{e:markov}
\frac{1}{\kappa_n}\vol\bigl(\Pi^*f\bigr)=\int_{\s^{n-1}}\frac{1}{|v|_{_{\Pi^*f}}^n}\,\dlat\sigma(v)
\geq \frac{1}{t^n}\,\sigma\left(v\in\s^{n-1}:
\frac{1}{|v|_{_{\Pi^*f}}^n}\geq\frac{1}{t^n}\right).
\end{equation*}
Then, taking
\[
s=\frac{t\,\kappa_n^{(n-1)/n}}{2\,\|f\|_{_{\frac{n}{n-1}}}\kappa_{n-1}}
\]
and using \eqref{e:AffineSobolev}, we get
\begin{equation*}
\sigma\left(v\in\s^{n-1}:|v|_{_{\Pi^*f}}\leq\frac{2s\|f\|_{_{\frac{n}{n-1}}}\kappa_{n-1}}{\kappa_n^{(n-1)/n}}\right)\leq
s^n.
\end{equation*}
This inequality for $s=1/2$ jointly with \eqref{e:f_n/n-1} and
\eqref{e:vPi^*f} yield
\begin{equation*}
\sigma\left(v\in\s^{n-1}: \w\left(P_{H_{v,0}} K\right)\leq
C_n\w\left(\sqrt{\frac{n}{n-1}}\,K\right)^{(n-1)/n}\right) \leq
\frac{1}{2^n},
\end{equation*}
which implies \eqref{e:Wills_sigma}, from the monotonicity (and the
translation invariance) of the classical Wills functional.

The last assertion follows from the fact that both $C_n$ and
$D_n/\sqrt{n}$ are convergent to $1/2$ and $\sqrt{2\pi/e}$, respectively,
as may be seen by using Stirling's formula and the value of $\kappa_n$.
\end{proof}

We note that Theorem~\ref{t:Wills_max_min_sigma} holds true in the general
setting of the generalized Wills functional $\w_u(\,\cdot\,,E)$, but in
that case, the bounds are given in terms of the functional
$\w_{pu}(\,\cdot\,,E)$ for the suitable $p\geq 1$ (cf.~\eqref{e:|f_u|_p}).
We have settled the result for the classical functional because in this
case the bounds are given in terms of $\w(\,\cdot)\,$ itself.

\subsection{Rogers-Shephard inequalities for the classical Wills functional}

The classical Rogers-Shephard inequality for the difference body states
that
\[
\vol(K-K)\leq \binom{2n}{n}\vol(K)
\]
(see e.g. \cite[Theorem~10.1.4]{Sch}). A strengthening of this inequality
was conjectured independently by Godbersen and Makai~Jr., namely, that the
mixed volume $\V\bigl(K[i],-K[n-i]\bigr)\leq\binom{n}{i}\vol(K)$ (see
\cite[Note~5 for Section~10.1]{Sch} and the references therein). Engaging
progresses have been made recently on this conjecture in \cite{AEFO}. Also
the corresponding upper bounds for the intrinsic volumes $\V_i(K-K)$,
$i=1,\dots,n-1$, are still unknown.

This subsection is devoted to studying Rogers-Shephard type inequalities
for the classical Wills functional. We will provide two different upper
bounds for $\w(K-K)$, which are obtained by using distinct techniques (we
will exploit either the difference function, or a Rogers-Shephard type
inequality for a log-concave function). These bounds will be not
comparable in the sense that, depending on the dimension, one is better
than the other.

First we prove Theorem \ref{t:R-S_Wills_K,L}. Indeed, profiting from
\eqref{e:diff_function}, we get the following more general result for two
convex bodies $K,L\in\K^n$.
\begin{theorem}\label{t:RS_Wills_K,L_1}
Let $K,L\in\K^n$ and let $\lambda\in(0,1)$. Then
\[
\w\left(\frac{(\lambda K)
\cap\bigl((1-\lambda)L\bigr)}{\sqrt{\lambda(1-\lambda)}}\right)
\w\bigl((1-\lambda)K-\lambda L\bigr)
\leq\frac{1}{\bigl(\lambda(1-\lambda)\bigr)^{n/2}}\w(K)\w(L).
\]
\end{theorem}

When $\lambda=1/2$, we obtain Theorem~\ref{t:R-S_Wills_K,L}.

\begin{proof}[Proof of Theorem \ref{t:RS_Wills_K,L_1}]
Let $f=\f{K}{B_2^n}{\u}\bigl((1-\lambda)\,\cdot\bigr)$ and
$g=\f{L}{B_2^n}{\u}(\lambda\,\cdot)$. Using Lemma \ref{lem:Asplund} as
well as the basic properties of the Asplund product we see~that the
$\lambda$-difference function associated to $f$ and $g$ can be written as
\begin{equation}\label{e:delta_lambda_fg}
\begin{split}
\Delta_{\lambda}^{f,g}(z) & =\sup_{z=x+y}
    f\left(\frac{x}{(1-\lambda)^2}\right)^{1-\lambda}g\left(\frac{-y}{\lambda^2}\right)^{\lambda}\\
 & =\sup_{z=x+y}e^{-\pi(1-\lambda)d\left(\frac{x}{1-\lambda},K\right)^2}
    e^{-\pi\lambda d\left(\frac{y}{\lambda},-L\right)^2}\\
 & =\sup_{z=x+y}e^{-\frac{\pi}{1-\lambda}\,d\bigl(x,(1-\lambda)K\bigr)^2}
    e^{-\frac{\pi}{\lambda}\,d\bigl(y,-\lambda L\bigr)^2}\\
 & =\biggl(\left(e^{-\frac{\pi}{1-\lambda}|\,\cdot\,|^2}\star\chi_{_{(1-\lambda)K}}\right)\star
    \left(e^{-\frac{\pi}{\lambda}|\,\cdot\,|^2}\star\chi_{_{-\lambda L}}\right)\biggr)(z)\\
 & =\left(e^{-\frac{\pi}{1-\lambda}|\,\cdot\,|^2}\star
    e^{-\frac{\pi}{\lambda}|\,\cdot\,|^2}\star\chi_{_{(1-\lambda)K-\lambda L}}\right)(z).
\end{split}
\end{equation}
Since, for any $v\in\R^n$,
\begin{equation}\label{e:delta_lambda_fg_2}
\begin{split}
\left(e^{-\frac{\pi}{1-\lambda}|\,\cdot\,|^2}\star
    e^{-\frac{\pi}{\lambda}|\,\cdot\,|^2}\right)(v) &
    =\sup_{w\in\R^n}e^{-\pi\left(\frac{|w|^2}{1-\lambda}+\frac{|v-w|^2}{\lambda}\right)}\\
 & =e^{-\frac{\pi}{\lambda(1-\lambda)}\inf_{w\in\R^n}\left(\lambda|w|^2+(1-\lambda)|v-w|^2\right)}\\
 & =e^{-\frac{\pi}{\lambda(1-\lambda)}
    \inf_{w\in\R^n}\left(|w|^2+(1-\lambda)|v|^2-2(1-\lambda)\langle v,w\rangle\right)}\\
 & =e^{-\frac{\pi}{\lambda(1-\lambda)}\inf_{r\geq 0}
    \inf_{|w|=r|v|}\left(r^2|v|^2+(1-\lambda)|v|^2-2(1-\lambda)\langle v,w\rangle\right)}\\
 & =e^{-\frac{\pi}{\lambda(1-\lambda)}\inf_{r\geq 0}
    \left(r^2+1-\lambda-2r(1-\lambda)\right)|v|^2}\\
 & =e^{-\frac{\pi}{\lambda(1-\lambda)}\left(1-\lambda-(1-\lambda)^2\right)|v|^2}
    =e^{-\pi|v|^2},
\end{split}
\end{equation}
then \eqref{e:delta_lambda_fg} yields
$\Delta_{\lambda}^{f,g}=e^{-\pi|\,\cdot\,|^2}\star\chi_{_{(1-\lambda)K-\lambda
L}}$. Thus, by Lemma \ref{lem:Asplund},
$\Delta_{\lambda}^{f,g}=\f{(1-\lambda)K-\lambda L}{B_2^n}{\u}$ and hence
\[
\int_{\R^n}\Delta_{\lambda}^{f,g}(x)\,\dlat x=\w\bigl((1-\lambda)K-\lambda
L\bigr).
\]
Moreover, we clearly have that
\[
\begin{split}
f(x)^{\lambda}g(x)^{1-\lambda}
 & =e^{-\pi\lambda d\bigl((1-\lambda)x,K\bigr)^2}e^{-\pi(1-\lambda)d(\lambda x,L)^2}\\
 & =e^{-\pi\lambda(1-\lambda)^2 d\left(x,\frac{K}{1-\lambda}\right)^2}
    e^{-\pi(1-\lambda)\lambda^2d\left(x,\frac{L}{\lambda}\right)^2}\\
 & \geq e^{-\pi\lambda(1-\lambda)^2d\left(x,\frac{K}{1-\lambda}\cap\frac{L}{\lambda}\right)^2}
    e^{-\pi(1-\lambda)\lambda^2d\left(x,\frac{K}{1-\lambda}\cap\frac{L}{\lambda}\right)^2}\\
 & =e^{-\pi d\left(\sqrt{\lambda(1-\lambda)} x,\frac{\left(\lambda K\right)
    \cap\left((1-\lambda)L\right)}{\sqrt{\lambda(1-\lambda)}}\right)^2},
\end{split}
\]
and then, applying the change of variable $y = \sqrt{\lambda(1-\lambda)}
x$, we get
\[
\int_{\R^n}f(x)^{\lambda}g(x)^{1-\lambda}\,\dlat x\geq
\frac{1}{\bigl(\lambda(1-\lambda)\bigr)^{n/2}} \w\left(\frac{(\lambda
K)\cap((1-\lambda)L)}{\sqrt{\lambda(1-\lambda)}}\right).
\]
Therefore, by \eqref{e:diff_function} and doing the changes of variable
$y=(1-\lambda)x$ and $y=\lambda x$, respectively, we obtain
\[
\begin{split}
\w\left(\frac{(\lambda K)
    \cap((1-\lambda)L)}{\sqrt{\lambda(1-\lambda)}}\right)\w\left((1-\lambda)K-\lambda L\right)\\
\leq\bigl(\lambda(1-\lambda)\bigr)^{n/2} \int_{\R^n}f(x)\,\dlat x
    \int_{\R^n}g(y)\,\dlat y
 & =\frac{1}{\bigl(\lambda(1-\lambda)\bigr)^{n/2}}\w(K)\w(L).
\end{split}
\]
This concludes the proof.
\end{proof}

It is known (see \cite[Theorem~2.2]{AGJV} as well as \cite{Co} for
related inequalities) that for a log-concave function
$f:\R^n\longrightarrow\R_{\geq0}$,
\begin{equation}\label{eq:extensionCol_f}
\int_{\R^n}\bigl(f\star\bar{f}\bigr)(x)\,\dlat x\leq\binom{2n}{n}
\|f\|_\infty \int_{\R^n}f(x)\,\dlat x,
\end{equation}
where $\bar{f}:\R^n\longrightarrow\R_{\geq0}$ is given by
$\bar{f}(x)=f(-x)$.

We conclude the paper by using this result to obtain the last announced
upper bound for the classical Wills functional of the difference body.
\begin{theorem}
Let $K\in\K^n$. Then
\begin{equation}\label{e:R-S_K-K3}
\w(K-K)\leq \frac{\binom{2n}{n}}{2^{n/2}}\,\w\left(\sqrt{2}\,K\right).
\end{equation}
\end{theorem}

\begin{proof}
Using \eqref{e:delta_lambda_fg_2} for $\lambda=1/2$, namely
$e^{-2\pi|\,\cdot\,|^2}\star
e^{-2\pi|\,\cdot\,|^2}=e^{-\pi|\,\cdot\,|^2}$, we obtain that
$e^{-\pi|\,\cdot\,|^2}\star
e^{-\pi|\,\cdot\,|^2}=e^{-\frac{\pi}{2}|\,\cdot\,|^2}$ and thus,
by Lemma~\ref{lem:Asplund},
\[
\f{K}{B_2^n}{\u}\star\bar{f}_{_{\!K,B_2^n}}^{\,\u}
    =\left(e^{-\pi|\,\cdot\,|^2}\star\chi_{_{K}}\right)\star\left(e^{-\pi|\,\cdot\,|^2}\star\chi_{_{-K}}\right)
    =e^{-\frac{\pi}{2}|\,\cdot\,|^2}\star\chi_{_{K-K}}=\f{K-K}{B_2^n}{\u/2}.
\]
Since
\[
\begin{split}
\int_{\R^n}\f{K-K}{B_2^n}{\u/2}(x)\,\dlat x &
    =\int_{\R^n}e^{-\frac{\pi}{2}d(x,K-K)^2}\,\dlat x
    =\int_{\R^n}e^{-\pi d\left(\frac{x}{\sqrt{2}},\frac{K-K}{\sqrt{2}}\right)^2}\,\dlat x\\
 & =2^{n/2}\,\w\left(\frac{K-K}{\sqrt{2}}\right),
\end{split}
\]
then \eqref{eq:extensionCol_f} yields
\[
\w\left(\frac{K-K}{\sqrt{2}}\right)\leq
\frac{\binom{2n}{n}}{2^{n/2}}\,\w(K),
\]
which concludes the proof.
\end{proof}

\begin{remark}
We observe that the bounds \eqref{e:R-S_K-K1} and \eqref{e:R-S_K-K3} are
not comparable. For instance, if we consider the cube $K=[0,1/2]^n$, for
which $\W_i(K)=\kappa_i/2^{n-i}$, $i=0,\dots,n$, then it is easy to check
that
\[
2^n\,\w(2K)=4^n<\binom{2n}{n}\sum_{i=0}^n\binom{n}{i}\frac{1}{2^{(n+i)/2}}
=\frac{\binom{2n}{n}}{2^{n/2}}\,\w\left(\sqrt{2}\,K\right)
\]
for $n=9$, whereas we get the opposite inequality when $n=3$.
\end{remark}


\begin{thebibliography}{99}
\bibitem{AA} D. Alonso-Guti\'errez, S. Artstein-Avidan, B.
Gonz\'alez, C. H. Jim\'enez, R. Villa, Rogers-Shephard and local
Loomis-Whitney type inequalities, {\it To appear in Math. Ann.},
\url{https://doi.org/10.1007/s00208-019-01834-3}.

\bibitem{ABG} D. Alonso-Guti\'errez, J. Bernu\'es, B. Gonz\'alez,
Zhang's inequality for log-concave functions, {\it To appear in GAFA
Seminar Notes}.

\bibitem{ABG2} D. Alonso-Guti\'errez, J. Bernu\'es, B. Gonz\'alez,
An extension of Berwald's inequality and its relation to Zhang's
inequality, {\it Preprint, arXiv:1908.01154}.

\bibitem{AGJV} D. Alonso-Guti\'errez, B. Gonz\'alez, C. H.
Jim\'enez, R. Villa, Rogers-Shephard inequality for log-concave functions,
{\it J. Funct. Anal.} {\bf 271} (11) (2016), 3269--3299.

\bibitem{AEFO} S. Artstein-Avidan, K. Einhorn, D. I. Florentin, Y.
Ostrover, On Godbersen's conjecture, {\it Geom. Dedicata} {\bf 178}
(2015), 337--350.

\bibitem{AKM} S. Artstein-Avidan, B. Klartag, V. Milman, The
Santal\'o point of a function, and a functional form of the
Santal\'o inequality, {\it Mathematika} {\bf 51} (1-2) (2014),
33--48.

\bibitem{B} K. Ball, Volumes of sections of cubes and related problems,
{\it Geometric Aspects of Functional Analysis}, Lecture Notes in Math.
1376, 251--260, Springer, Berlin, 1989.

\bibitem{Ber} L. Berwald: Verallgemeinerung eines Mittelwertsatzes von J.
Favard f\"ur positive konkave Funktionen, {\it Acta Math.} {\bf 79}
(1947), 17--37.

\bibitem{BH93} U. Betke, M. Henk, Intrinsic volumes and lattice points
of crosspolytopes, {\it Monatsh. Math.} {\bf 115} (1-2) (1993),
27--33.

\bibitem{Cohn} D. L. Cohn, {\it Measure theory}, 2nd revised ed.,
Birkh\"auser/Springer, New York Heidelberg, 2013.

\bibitem{Co} A. Colesanti, Functional inequalities related to the
Rogers-Shephard inequality, {\it Mathematika} {\bf 53} (2006), 81--101.

\bibitem{CSY} A. Colesanti, E. Saor\'in G\'omez, J. Yepes Nicol\'as,
On a linear refinement of the Pr\'ekopa-Leindler inequality, {\it Canad.
J. Math.} {\bf 68} (4) (2016), 762--783.

\bibitem{G} R. J. Gardner, The Brunn-Minkowski inequality, {\it Bull. Amer.
Math. Soc.} {\bf 39} (3) (2002), 355--405.

\bibitem{Gr} P. M. Gruber, {\it Convex and Discrete Geometry},
Springer, Berlin Heidelberg, 2007.

\bibitem{Ha75} H. Hadwiger, Das Wills'sche Funktional, {\it Monatsh. Math.}
{\bf 79} (1975), 213--221.

\bibitem{Ha79} H. Hadwiger, Gitterpunktanzahl im Simplex und Wills'sche
Vermutung, {\it Math. Ann.} {\bf 239} (3) (1979), 271--288.

\bibitem{HCY13} M. A. Hern\'andez Cifre, J. Yepes Nicol\'as, On the
roots of the Wills functional, {\it J. Math. Anal. Appl.} {\bf
401} (2013), 733--742.

\bibitem{HCY15} M. A. Hern\'andez Cifre, J. Yepes Nicol\'as, On the
roots of generalized Wills $\mu$-polynomials, {\it Rev. Mat.
Iberoamericana} {\bf 31} (2) (2015), 477--496.

\bibitem{K} J. Kampf, On weighted parallel volumes, {\it Beitr\"age Algebra
Geom.} {\bf 50} (2) (2009), 495--519.

\bibitem{KM} B. Klartag, V. D. Milman, Geometry of log-concave functions
and measures, {\it Geom. Dedicata} {\bf 112} (2005), 169--182.

\bibitem{Mc75} P. McMullen, Non-linear angle-sum relations for
polyhedral cones and polytopes, {\it Math. Proc. Cambridge Philos.
Soc.} {\bf 78} (1975), 247--261.

\bibitem{Mc91} P. McMullen, Inequalities between intrinsic volumes, {\it
Monatsh. Math.} {\bf 111} (1) (1991), 47--53.

\bibitem{Mi} V. D. Milman, In\'egalit\'e de Brunn-Minkowski inverse et
applications \`a la th\'eorie locale des espaces norm\'es (An inverse form
of the Brunn-Minkowski inequality, with applications to the local theory
of normed spaces), {\it C. R. Acad. Sci. Paris Ser. I Math.} {\bf 302} (1)
(1986), 25--28.

\bibitem{SY93} J. R. Sangwine-Yager, Mixed volumes. In: {\it Handbook of
Convex Geometry} (P. M. Gruber and J. M. Wills eds.),
North-Holland, Amsterdam, 1993, 43--71.

\bibitem{Sch} R. Schneider, {\it Convex bodies: The Brunn-Minkowski
theory}, 2nd expanded ed. Encyclopedia of Mathematics and its Applications
151, Cambridge, Cambridge University Press, 2014.

\bibitem{Vi96} R. A. Vitale, The Wills functional and Gaussian processes,
{\it Ann. Probab.} {\bf 24} (4) (1996), 2172--2178.

\bibitem{Vi99} R. A. Vitale, A log-concavity proof for a Gaussian
exponential bound, {\it Contemp. Math.} {\bf 239} (1999),
209--212.

\bibitem{Vi01} R. A. Vitale, Intrinsic volumes and Gaussian processes,
{\it Adv. Appl. Prob. (SGSA)} {\bf 33} (2001), 354--364.

\bibitem{Vi08} R. A. Vitale, Y. Wang, The Wills functional for Poisson
processes, {\it Statist. Probab. Lett.} {\bf 78} (14) (2008), 2181--2187.

\bibitem{Vi19} R. A. Vitale, On an exponential functional for
Gaussian processes and its geometric foundations, {\it J. Math.
Sci.} {\bf 238} (4) (2019), 406--414.

\bibitem{W73} J. M. Wills, Zur Gitterpunktanzahl konvexer Mengen, {\it
Elem. Math.} {\bf 28} (1973), 57--63.

\bibitem{W75} J. M. Wills, Nullstellenverteilung zweier konvexgeometrischer
Polynome, {\it Beitr\"age Algebra Geom.} {\bf 29} (1989), 51--59.

\bibitem{W90} J. M. Wills, Minkowski's successive minima and the zeros of a
convexity-function, {\it Monatsh. Math.} {\bf 109} (2) (1990), 157--164.

\bibitem{Z} G. Zhang, The affine Sobolev inequality, {\it J. Differ. Geom.}
{\bf 53} (1) (1999), 183--202.

\end{thebibliography}
\end{document}